\documentclass{amsart}
\usepackage{amssymb}
\usepackage{mathrsfs}
\usepackage[all]{xy}
\usepackage{tikz}
\usepackage{hyperref}
\usetikzlibrary{matrix,arrows}

\newcommand{\bk}{\Bbbk}
\newcommand{\C}{\mathbb{C}}
\newcommand{\F}{\mathbb{F}}
\newcommand{\Z}{\mathbb{Z}}

\newcommand{\D}{\mathrm{D}}
\newcommand{\Db}{\D^{\mathrm{b}}}
\newcommand{\Kb}{\mathrm{K}^{\mathrm{b}}}

\newcommand{\Sh}{\mathrm{Sh}}

\newcommand{\cG}{\mathcal{G}}
\newcommand{\cT}{\mathcal{T}}
\newcommand{\cB}{\mathcal{B}}
\newcommand{\cP}{\mathcal{P}}
\newcommand{\cU}{\mathcal{U}}

\newcommand{\sA}{\mathscr{A}}
\newcommand{\sB}{\mathscr{B}}
\newcommand{\sC}{\mathscr{C}}
\newcommand{\bDelta}{\bar{\Delta}}
\newcommand{\bnabla}{\bar{\nabla}}
\newcommand{\fDelta}{\mathscr{F}(\Delta)}
\newcommand{\fnabla}{\mathscr{F}(\nabla)}
\newcommand{\fbDelta}{\mathscr{F}(\bDelta)}
\newcommand{\fbnabla}{\mathscr{F}(\bnabla)}
\newcommand{\Irr}{\mathrm{Irr}}
\newcommand{\il}{\imath^{\mathrm{L}}}
\newcommand{\ir}{\imath^{\mathrm{R}}}
\newcommand{\pil}{\Pi^{\mathrm{L}}}
\newcommand{\pir}{\Pi^{\mathrm{R}}}

\newcommand{\Gv}{G^\vee}
\newcommand{\Bv}{B^\vee}
\newcommand{\Tv}{T^\vee}
\newcommand{\Uv}{U^\vee}
\newcommand{\Cent}{\mathrm{Z}}

\newcommand{\Go}{{G_{\mathbf{O}}}}
\newcommand{\Gr}{\mathcal{G}r}
\newcommand{\bX}{\mathbf{X}}
\newcommand{\bXp}{\mathbf{X}^+}

\newcommand{\ubk}{\underline{\bk}}

\newcommand{\Perv}{\mathrm{Perv}}
\newcommand{\Tilt}{\mathrm{Tilt}}
\newcommand{\Parity}{\mathrm{Parity}}
\newcommand{\ParityB}{\mathrm{Parity}_{(\cB)}}
\newcommand{\ParityGo}{\mathrm{Parity}_{(\Go)}}
\newcommand{\Dbom}{\Db_{\Omega}}
\newcommand{\Dbb}{\Db_{(\cB)}}
\newcommand{\Dbgo}{\Db_{(\Go)}}

\newcommand{\DGo}{\Db_{(\Go)}(\Gr,\bk)}
\newcommand{\PervGo}{\Perv_\Go(\Gr,\bk)}
\newcommand{\Sat}{\mathcal{S}}
\newcommand{\Evsat}{\mathbb{S}}
\newcommand{\Pvsat}{\hat{\mathbb{S}}}

\newcommand{\IC}{\mathrm{IC}}
\newcommand{\cE}{\mathcal{E}}
\newcommand{\cIstd}{\mathcal{I}_!}
\newcommand{\cIcos}{\mathcal{I}_*}

\newcommand{\Ad}{\mathrm{Ad}}
\newcommand{\cN}{{\mathcal{N}}}
\newcommand{\tcN}{{\tilde{\cN}}}
\newcommand{\cO}{\mathcal{O}}
\newcommand{\cON}{\mathcal{O}_\cN}
\newcommand{\reg}{{\mathrm{reg}}}
\newcommand{\Dist}{\mathrm{Dist}}
\newcommand{\fe}{\mathsf{e}}
\newcommand{\Lie}{\mathrm{Lie}}

\newcommand{\irr}{\mathrm{L}}
\newcommand{\til}{\mathrm{T}}
\newcommand{\wey}{\mathrm{M}}
\newcommand{\cow}{\mathrm{N}}

\newcommand{\Coh}{\mathrm{Coh}}
\newcommand{\Gm}{{\mathbb{G}_{\mathrm{m}}}}

\newcommand{\CohGm}{\Coh^{\Gv \times \Gm}}
\newcommand{\CohGmN}{{\Coh^{\Gv \times \Gm}(\cN)}}
\newcommand{\CohN}{\Coh(\cN)}
\newcommand{\PCohGmN}{\mathrm{P}\CohGmN}
\newcommand{\PCohN}{\mathrm{P}\CohN}

\newcommand{\Rep}{\mathrm{Rep}}

\newcommand{\cH}{\mathcal{H}} 
\newcommand{\cF}{\mathcal{F}}

\newcommand{\simto}{\overset{\sim}{\to}}
\newcommand{\la}{\langle}
\newcommand{\ra}{\rangle}

\newcommand{\pt}{\mathrm{pt}}

\newcommand{\aq}{\mathbin{/\!/}}
\newcommand{\bigaq}{\mathbin{\big/\!\big/}}

\DeclareMathOperator{\chr}{char}
\DeclareMathOperator{\End}{End}
\DeclareMathOperator{\Hom}{Hom}
\DeclareMathOperator{\Ext}{Ext}
\DeclareMathOperator{\cok}{cok}
\newcommand{\uHom}{\underline{\mathrm{Hom}}}
\newcommand{\uEnd}{\underline{\mathrm{End}}}
\newcommand{\uExt}{\underline{\mathrm{Ext}}}
\newcommand{\Lotimes}{\mathchoice%
  {\overset{\scriptscriptstyle L}{\otimes}}%
  {\otimes^{\scriptscriptstyle L}}{\otimes^L}{\otimes^L}}

\DeclareMathOperator{\tdim}{tdim}

\numberwithin{equation}{section}
\newtheorem{thm}{Theorem}[section]
\newtheorem*{thm*}{Theorem}
\newtheorem{lem}[thm]{Lemma}
\newtheorem{prop}[thm]{Proposition}
\newtheorem{cor}[thm]{Corollary}
\theoremstyle{definition}
\newtheorem{defn}[thm]{Definition}

\theoremstyle{remark}
\newtheorem{rmk}[thm]{Remark}

\title[Parity sheaves on the affine Grassmannian]{Parity sheaves on the affine Grassmannian and\\ the Mirkovi\'c--Vilonen conjecture}

\author{Pramod N. Achar}
\address{Department of Mathematics\\
  Louisiana State University\\
  Baton Rouge, LA 70803\\
  U.S.A.}
\email{pramod@math.lsu.edu}

\author{Laura Rider}
\address{Department of Mathematics\\
  Massachusetts Institute of Technology\\
  Cambridge, MA 02139\\
  U.S.A.}
\email{laurajoy@mit.edu}

\thanks{P.A. was supported by NSF Grant No.~DMS-1001594. L.R. was supported by an NSF Postdoctoral Research Fellowship.}

\begin{document}

\begin{abstract} We prove the Mirkovi\'c--Vilonen conjecture: the integral local intersection cohomology groups of spherical Schubert varieties on the affine Grassmannian have no $p$-torsion, as long as $p$ is outside a certain small and explicitly given set of prime numbers. 
(Juteau has exhibited counterexamples when $p$ is a bad prime.)
The main idea is to convert this topological question into an algebraic question about perverse-coherent sheaves on the dual nilpotent cone using the 
Juteau--Mautner--Williamson 
theory of parity sheaves.
\end{abstract}

\maketitle

\section{Introduction}
\label{sect:intro}

\subsection{Overview}

Let $G$ be a connected complex reductive group, and let $\Gr$ denote its affine Grassmannian.  This space has the remarkable property that its topology encodes the representation theory of the split Langlands dual group $\Gv$ over any field $\bk$ (or even over a commutative ring).  To be more precise, the \emph{geometric Satake equivalence}, in the form due to Mirkovi\'c--Vilonen~\cite{mv} 
(see also~\cite{lusztig,gin:pslg}), asserts that there is an equivalence of tensor categories
\begin{equation}\label{eqn:satake}
\Sat: \Rep(\Gv) \simto \Perv_\Go(\Gr,\bk)
\end{equation}
where $\Perv_\Go(\Gr,\bk)$ is the category of spherical perverse $\bk$-sheaves on $\Gr$.  (A full explanation of the notation is given in~\S\ref{subsect:notation} below.)  This result raises the possibility of comparing representation theory over different fields via the universal coefficient theorem of topology.

For instance, let $\lambda$ be a dominant coweight for $G$, and let $\cIstd(\lambda,\bk)$ denote the ``standard" perverse sheaf on the corresponding stratum of $\Gr$.  This perverse sheaf serves as a topological realization of a Weyl module for $\Gv$.  When $\bk = \C$, it is simple, and its stalks are described by Kazhdan--Lusztig theory.

 With a view to applications in modular representation theory, Mirkovi\'c and Vilonen conjectured in the late 1990s~\cite{mv1} that that the stalks of $\cIstd(\lambda, \Z)$ are torsion-free. This implies that the $\bk$-stalks are ``independent'' of $\bk$. Their conjecture was slightly too optimistic: counterexamples due to Juteau~\cite{juteau} reveal the presence of torsion, but only at bad primes.  Juteau proposed a modified conjecture, asserting that there is no $p$-torsion as long as $p$ is a good prime for $\Gv$.  
In this paper, we prove the following result, confirming this conjecture in nearly all cases.

\begin{thm}\label{thm:mv-intro}
If $p$ is a JMW prime for $\Gv$ {\normalfont (see Table~\ref{tab:jmw})}, then the stalks of $\cIstd(\lambda,\Z)$ have no $p$-torsion.  Similarly, if $\bk$ is a field whose characteristic is a JMW prime, then the stalks of $\cIstd(\lambda,\bk)$ have a parity-vanishing property.
\end{thm}

\begin{table}[t]
\begin{center}
\begin{tabular}[t]{ l | c }
 Type & Bound \\
 \hline 
    $A_n$ & any $p$ \\
    $B_n, D_n$& $p>2$ \\
    $C_n$ & $p>n$ \\
 \hline
\end{tabular}
\qquad
\begin{tabular}[t]{ l | c }
 Type & Bound \\
 \hline 
    $E_6, F_4, G_2$& $p>3$\\
    $E_7$& $p>19$\\
    $E_8$& $p>31$\\
  \hline  
\end{tabular}
\bigskip
\end{center}
\caption{Currently known bounds$^1$ for JMW primes}\label{tab:jmw}
\end{table}

An outline of the proof will be explained below, after some preliminaries. In a subsequent paper~\cite{ard}, the authors exploit this result to establish a modular analogue of the derived equivalence of~\cite[Theorem~9.4.3]{abg}. 

\subsection{The constructible side}
\label{subsect:notation}

Recall that $\Gr = G_{\mathbf{K}}/\Go$, where $\mathbf{K} = \C((t))$ and $\mathbf{O} = \C[[t]]$.  For the remainder of the introduction, $\bk$ will denote an algebraically closed field.  Let $\DGo$ denote the bounded derived category of complexes of $\bk$-sheaves on $\Gr$ that are constructible with respect to the $\Go$-orbits, and let $\PervGo \subset \DGo$ be the subcategory of perverse sheaves.  Those $\Go$-orbits are naturally in bijection with the set $\bXp$ of dominant coweights for $G$.  For $\lambda \in \bXp$, let $i_\lambda: \Gr_\lambda \hookrightarrow \Gr$ be the inclusion map of the corresponding orbit.

For $\lambda \in \bXp$, the irreducible (resp.~Weyl, dual Weyl, indecomposable tilting) $\Gv$-module of highest weight $\lambda$ is denoted by $\irr(\lambda)$ (resp.~$\wey(\lambda)$, $\cow(\lambda)$, $\til(\lambda)$).  The perverse sheaves corresponding to these objects under $\Sat$ are denoted by $\IC(\lambda,\bk)$, (resp.~$\cIstd(\lambda,\bk)$, $\cIcos(\lambda,\bk)$, $\cT(\lambda,\bk)$).  Of course, $\IC(\lambda,\bk)$ is a simple perverse sheaf.  We saw $\cIstd(\lambda,\bk)$ earlier; $\cIcos(\lambda,\bk)$ is its Verdier dual, a costandard perverse sheaf.

What about the $\cT(\lambda,\bk)$?  It is a deep insight of Juteau--Mautner--Williamson that these perverse sheaves should be characterized by a topological property: specifically, they ought to be \emph{parity sheaves} in the sense of~\cite{jmw}.

\begin{defn}\label{defn:jmw}
A prime number $p$ is said to be a \emph{JMW prime} for $\Gv$ if it is good for $\Gv$ and, whenever $\bk$ has characteristic $p$, each $\cT(\lambda,\bk)$ is a parity sheaf on $\Gr$.
\end{defn}

Juteau, Mautner, and Williamson have shown in~\cite[Theorem~1.8]{jmw:parity} that for quasisimple $\Gv$, the primes in Table~\ref{tab:jmw} are JMW primes.  They conjecture\footnote{Since this paper appeared in preprint form, Mautner and Riche have proved that every good prime is a JMW prime~\cite{mr}, confirming~\cite[Conjecture~1.10]{jmw:parity}.} that every good prime is JMW (see~\cite[Conjecture~1.10]{jmw:parity}).  

\subsection{The coherent side}
\label{subsect:proof}

The main idea of the proof of Theorem~\ref{thm:mv-intro} is to translate the problem into an algebraic question about coherent sheaves on the nilpotent cone $\cN$ for $\Gv$.  The motivation comes from an old result of Ginzburg~\cite[Proposition~1.10.4]{gin:pslg}: when $\bk = \C$, he showed that for all $V_1, V_2 \in \Rep(\Gv)$, there is an isomorphism of graded vector spaces
\begin{equation}\label{eqn:ginzburg-intro}
\Hom^\bullet_{\DGo}(\Sat(V_1), \Sat(V_2)) \cong \Hom^\bullet_{\CohGmN}(V_1 \otimes \cON, V_2 \otimes \cON).
\end{equation}
For details on the category $\CohGmN$, see~\S\ref{subsect:pervcoh}.

To imitate this in positive characteristic, we need control over the algebraic geometry of $\cN$. To this end, we impose the following condition on $\Gv$:
\begin{equation}\label{eqn:reasonable}
\begin{minipage}{3.6in}
The derived group of $\Gv$ is simply connected, and its Lie
algebra admits a nondegenerate $\Gv$-invariant bilinear form.
\end{minipage}
\end{equation}
This condition holds for $\mathrm{GL}(n)$ and for every simply-connected quasi-simple group that is not of type $A$.  See~\cite[\S2.9]{jan:nort} for a discussion of other situations in which this holds.  Under this condition, it is feasible to adapt Ginzburg's argument, provided that $\Sat(V_1)$ and $\Sat(V_2)$ are parity.

To push this result further, we need the following observation: coherent sheaves of the form $V \otimes \cON$ also lie in the category of \emph{perverse-coherent sheaves}, denoted $\PCohGmN$, or simply $\PCohN$.  This category, which is the heart of a certain $t$-structure on $\Db\CohGmN$, looks very different from $\CohGmN$. For instance, every object of $\PCohN$ has finite length.  We will not use the details of its definition in this paper; we just require a structural property discussed in~\S\ref{subsect:pervcoh}.

Interpreting the right-hand side of~\eqref{eqn:ginzburg-intro} as a $\Hom$-group in $\PCohN$ leads to new avenues for generalizing that result.  For $\mu \in \bXp$, let $\PCohN_{\le \mu} \subset \PCohN$ be the Serre subcategory generated by $\cow(\nu) \otimes \cON\la n\ra$ with $\nu \le \mu$.  (Here, $\la n\ra$ indicates a twist of the $\Gm$-action.)  In~\S\ref{sect:extgr}, we prove the following result, which seems to be new even for $\bk = \C$.

\begin{thm}\label{thm:extgr-intro}
Assume that $\chr \bk$ is a JMW prime for $\Gv$, and that~\eqref{eqn:reasonable} holds for $\Gv$. Let $j: \Gr \setminus \overline{\Gr_\lambda} \to \Gr$ be the inclusion map,
and let $\Pi: \PCohN \to \PCohN/\PCohN_{\le \lambda}$ be the Serre quotient functor.
If $V_1 \in \Rep(\Gv)$ has a Weyl filtration and $V_2 \in \Rep(\Gv)$ has a good filtration, there is a natural isomorphism
\[
\Hom^\bullet(j^*\Sat(V_1), j^*\Sat(V_2)) \cong \Hom^\bullet(\Pi(V_1 \otimes \cON), \Pi(V_2 \otimes \cON)).
\]
\end{thm}

Intuitively, this theorem gives us an algebraic counterpart in $\Db\CohGmN$ of the geometric notion of ``restricting to an open subset'' in $\Gr$.  Once we have that, it is not difficult to translate the problem of studying stalks of $\cIstd(\lambda,\bk)$ into an algebraic question about certain objects in $\PCohN$ and its quotients.  The latter question turns out to be quite easy (see Lemma~\ref{lem:delta-nabla-free}).

\subsection{Outline of the paper}

In~\S\ref{sect:propstrat}, we recall the necessary background on properly stratified categories and on $\PCohN$, largely following the work of Minn-Thu-Aye.   We review the theory of parity sheaves in~\S\ref{sect:parity}.  In~\S\ref{sect:ginzburg}, which can be read independently of the rest of the paper, we study the cohomology of parity sheaves on flag varieties of Kac--Moody groups, generalizing earlier results of Soergel and Ginzburg.
That result is a step on the way to Theorem~\ref{thm:extgr-intro}, which is proved in~\S\ref{sect:extgr}.  Finally, the main result, Theorem~\ref{thm:mv-intro}, is proved in~\S\ref{sect:mv}. 

\subsection{Acknowledgements}

We are grateful to D.~Juteau, C.~Mautner, S.~Riche, and G.~Will\-iam\-son for helpful comments on a previous draft of this paper.

\section{Properly stratified categories}
\label{sect:propstrat}

\subsection{Definition and background}
\label{subsect:propstrat-defn}

Let $\bk$ be a field, and let $\sC$ be a $\bk$-linear abelian category in which every object has finite length. Assume that $\sC$ is equipped with an automorphism $\la 1\ra: \sC \to \sC$, which we will refer to as the \emph{Tate twist}. For $X, Y \in \sC$, let $\uHom(X,Y)$ be the graded vector space given by
\[
\uHom(X,Y)_n = \Hom(X,Y \la n\ra).
\]
The Tate twist induces an action of $\Z$ on the set $\Irr(\sC)$ of isomorphism classes of simple objects in $\sC$.  Assume that this action is free, and let $\Omega = \Irr(\sC)/\Z$.  For each $\gamma \in \Omega$, choose a representative simple object $L_\gamma \in \sC$ whose isomorphism class lies in the $\Z$-orbit $\gamma \subset \Irr(\sC)$.  Thus, every simple object in $\sC$ is isomorphic to some $L_\gamma\la n\ra$ with $\gamma \in \Omega$ and $n \in \Z$.

Assume that $\Omega$ is equipped with a partial order $\le$, and that for any $\gamma \in \Omega$, the set $\{ \xi \in \Omega \mid \xi \le \gamma \}$ is finite. For any order ideal $\Gamma \subset \Omega$, let $\sC_\Gamma \subset \sC$ be the Serre subcategory generated by the simple objects $\{ L_\gamma\la n\ra \mid \text{$\gamma \in \Gamma$, $n \in \Z$} \}$.  (Recall that an \emph{order ideal} is a subset $\Gamma \subset \Omega$ such that if $\gamma \in \Gamma$ and $\xi \le \gamma$, then $\xi \in \Gamma$.)  As a special case, we write
\begin{equation}\label{eqn:defn-serre-le}
\sC_{\le \gamma} = \sC_{\{ \xi \in \Omega \mid \xi \le \gamma \}}.
\end{equation}
The category $\sC_{< \gamma}$ is defined similarly.

\begin{defn}\label{defn:propstrat}
Suppose $\sC$, $\Omega$, and $\le$ are as above.  We say that $\sC$ is a \emph{graded properly stratified category} if for each $\gamma \in \Omega$, the following conditions hold:
\begin{enumerate}
\item We have $\End(L_\gamma) \cong \bk$.\label{it:scalar}
\item There is an object $\bDelta_\gamma$ and a surjective morphism $\phi_\gamma: \bDelta_\gamma \to L_\gamma$ such that\label{it:propstd}
\[
\ker(\phi_\gamma) \in \sC_{< \gamma}\qquad\text{and}\qquad
\text{$\uHom(\bDelta_\gamma,L_\xi) = \uExt^1(\bDelta_\gamma,L_\xi) = 0$ if $\xi \not\ge \gamma$.}
\]
\item There is an object $\bnabla_\gamma$ and an injective morphism $\psi_\gamma: L_\gamma \to \bnabla_\gamma$ such that\label{it:propcostd}
\[
\cok(\psi_\gamma) \in \sC_{< \gamma}\qquad\text{and}\qquad
\text{$\uHom(L_\xi,\bnabla_\gamma) = \uExt^1(L_\xi,\bnabla_\gamma) = 0$ if $\xi \not\ge \gamma$.}
\]
\item In $\sC_{\le \gamma}$, $L_\gamma$ admits a projective cover $\Delta_\gamma \to L_\gamma$.  Moreover, $\Delta_\gamma$ admits a filtration whose subquotients are of the form $\bDelta_\gamma\la n\ra$ for various $n \in \Z$.\label{it:standard}
\item In $\sC_{\le \gamma}$, $L_\gamma$ admits an injective envelope $L_\gamma \to \nabla_\gamma$.  Moreover, $\nabla_\gamma$ admits a filtration whose subquotients are of the form $\bnabla_\gamma\la n\ra$ for various $n \in \Z$.\label{it:costandard}
\item We have $\uExt^2(\Delta_\gamma,\bnabla_\xi) = \uExt^2(\bDelta_\gamma,\nabla_\xi) = 0$ for all $\gamma, \xi \in \Omega$. \label{it:ext2}
\end{enumerate}
An object in $\sC$ is said to be \emph{standard} (resp.~\emph{costandard}, \emph{proper standard}, \emph{proper costandard}) if it is isomorphic to some $\Delta_\gamma\la n\ra$ (resp.~$\nabla_\gamma\la n\ra$, $\bDelta_\gamma\la n\ra$, $\bnabla_\gamma\la n\ra$).

More generally, a \emph{standard} (resp.~\emph{costandard}, \emph{proper standard}, \emph{proper costandard}) {filtration} of an object of $\sC$ is a filtration whose subquotients are all standard (resp.~costandard, proper standard, proper costandard) objects.
\end{defn}

Routine arguments (see~\cite[Lemma~1]{bez:qes}) show that when objects $\bDelta_\gamma$, $\bnabla_\gamma$, $\Delta_\gamma$, $\nabla_\gamma$ with the above properties exist, they are unique up to isomorphism.  It may happen that $\bDelta_\gamma \cong \Delta_\gamma$ and $\bnabla_\gamma \cong \nabla_\gamma$; in that case, $\sC$ is usually called a \emph{highest weight} or \emph{quasi-hereditary category}.  The class of objects in $\sC$ admitting a standard (resp.~costandard, proper standard, proper costandard) filtration is denoted
\[
\fDelta, \qquad\text{resp.}\qquad
\fnabla, \qquad \fbDelta, \qquad \fbnabla.
\]

The relationship between the notions above and the notion of a \emph{properly stratified algebra}~\cite{dlab, fm} is explained in~\cite{myron}.  In particular, results in ~\cite{myron} explain how to transfer results from the literature on properly stratified algebras to our setting.  For instance, the following result is a restatement of~\cite[Definition~4 and Theorem~5]{dlab}.

\begin{prop}\label{prop:recollement}
Let $\Gamma \subset \Omega$ be a finite order ideal.  Then the Serre quotient $\sC / \sC_\Gamma$ is again a graded properly stratified category, and $\Irr(\sC / \sC_\Gamma)/\Z$ is naturally identified with $\Omega \setminus \Gamma$.  Indeed, we have a recollement diagram
\[
\begin{tikzpicture}
\node (a) at (0,0) {$\Db\sC_\Gamma$};
\node (b) at (4,0) {$\Db\sC$};
\node (c) at (8,0) {$\Db(\sC/\sC_\Gamma)$};
\draw[->, >=angle 90] (a) -- (b);
\draw[->, >=angle 90, above] (a) to node {$\imath$} (b);
\draw[->, >=angle 90, above] (b) to node {$\Pi$} (c);
\draw[->,>=angle 90, bend right=30, above] (b)  to node {$\il$} (a);
\draw[->,>=angle 90, bend left=30, above] (b) to node {$\ir$} (a);
\draw[->,>=angle 90, bend right=30, above]  (c) to node {$\pil$} (b);
\draw[->,>=angle 90, bend left=30, above]  (c) to node {$\pir$} (b);
\end{tikzpicture}
\]
\end{prop}

Here, the superscripts $\mathrm{L}$ and $\mathrm{R}$ indicate the left and right adjoints, respectively, of $\imath$ and $\Pi$.
An important property implied by the preceding proposition is that
\[
\uExt^k(\Delta_\gamma, \bnabla_\xi) = \uExt^k(\bDelta_\gamma,\nabla_\xi) = 0
\qquad\text{for all $k > 0$.}
\]
Also implicit in Proposition~\ref{prop:recollement} (or explicit in its proof) are the next two lemmas, which express the compatibility of the various functors with the properly stratified structure.  For analogues in the quasi-hereditary case, see~\cite{cps}.

\begin{lem}\label{lem:preserve-Tilt}
The functors $\imath$ and $\Pi$ are $t$-exact and preserve the property of having a standard (resp.~costandard, proper standard, proper costandard) filtration.
\end{lem}

The remaining functors in the recollement diagram are not $t$-exact in general, but they do send certain classes of objects to the heart of the $t$-structure.

\begin{lem}\label{lem:preserve-Filt}
The functors $\il$ and $\pil$ preserve the property of having a standard or proper standard filtration.  The functors $\ir$ and $\pir$ preserve the property of having a costandard or proper costandard filtration.
\end{lem}

\subsection{Tilting objects}
\label{subsect:tilting}

In contrast with the quasi-hereditary case, there are, in general, two inequivalent notions of ``tilting'' in a properly stratified category.

\begin{defn}
A \emph{tilting object} is an object in $\fDelta \cap \fbnabla$.  A \emph{cotilting object} is an object in $\fbDelta \cap \fnabla$.
\end{defn}

The next proposition gives the classification of indecomposable tilting and cotilting objects.  (See~\cite{ahlu} for a similar statement for properly stratified algebras.)

\begin{prop}\label{prop:tilt-classif}
For each $\gamma \in \Omega$, there is an indecomposable tilting object $T_\gamma$, unique up to isomorphism, that fits into short exact sequences
\[
0 \to \Delta_\gamma \to T_\gamma \to X \to 0
\qquad\text{and}\qquad
0 \to Y \to T_\gamma \to \bnabla_\gamma \to 0
\]
with $X \in \fDelta_{<\gamma}$ and $Y \in \fbnabla_{\le \gamma}$.  Dually, there is an indecomposable cotilting object $T'_\gamma$, unique up to isomorphism, with short exact sequences
\[
0 \to \bDelta_\gamma \to T'_\gamma \to X' \to 0
\qquad\text{and}\qquad
0 \to Y' \to T'_\gamma \to \nabla_\gamma \to 0
\]
with $X' \in \fbDelta_{\le \gamma}$ and $Y' \in \fnabla_{< \gamma}$.  Moreover, every indecomposable tilting (resp.~cotilting) object is isomorphic to some $T_\gamma\la n\ra$ (resp.~$T'_\gamma\la n\ra$).
\end{prop}

\begin{lem}\label{lem:tilt-cotilt}
Assume that the tilting and cotilting objects in $\sC$ coincide, i.e., that for each $\gamma \in \Omega$, there is an integer $m_\gamma$ such that $T_\gamma \cong T'_\gamma\la m_\gamma\ra$.  Then:
\begin{enumerate}
\item If $\gamma \in \Omega$ is minimal, then $\Delta_\gamma \cong T_\gamma \cong T'_\gamma\la m_\gamma\ra \cong \nabla_\gamma\la m_\gamma\ra$.\label{it:tilt-min}
\item For any $\gamma \in \Omega$, we have $\uExt^1(\nabla_\gamma, \bnabla_\gamma) =0$.\label{it:nabla-bnabla}
\item For any $\gamma \in \Omega$, there are natural isomorphisms\label{it:delta-nabla-hom}
\[
\uHom(\Delta_\gamma,\Delta_\gamma) \cong
\uHom(\Delta_\gamma,\nabla_\gamma\la m_\gamma\ra) \cong
\uHom(\nabla_\gamma,\nabla_\gamma).
\]
\end{enumerate}
\end{lem}
\begin{proof}
\eqref{it:tilt-min}~This is immediate from the short exact sequences in Proposition~\ref{prop:tilt-classif}.

\eqref{it:nabla-bnabla}~Consider the long exact sequence
\[
\cdots \to \uHom(Y',\bnabla_\gamma) \to \uExt^1(\nabla_\gamma, \bnabla_\gamma) \to \uExt^1(T'_\gamma, \bnabla_\gamma) \to \cdots
\]
The first term vanishes because $Y' \in \sC_{< \gamma}$, and the last term vanishes because $T'_\gamma \cong T_\gamma\la -m_\gamma\ra \in \fDelta$.  The result follows.

\eqref{it:delta-nabla-hom}~It is easy to see that the natural maps $\uHom(\Delta_\gamma,\Delta_\gamma) \to \uHom(\Delta_\gamma,T_\gamma)$ and $\uHom(\Delta_\gamma,T'_\gamma\la m_\gamma\ra) \to \uHom(\Delta_\gamma, \nabla_\gamma\la m_\gamma\ra)$ are both isomorphisms.  The proof of the second isomorphism in the statement is similar.
\end{proof}

\begin{prop}[{\cite{myron}; cf.~\cite[Proposition~1.5]{bbm}}]\label{prop:ringel-duality}
Assume that the tilting and cotilting objects in $\sC$ coincide.  Let $\cT \subset \sC$ be the full subcategory of tilting objects, and consider its homotopy category $\Kb\cT$.  The obvious functor
\begin{equation}\label{eqn:ringel}
\Kb\cT \to \Db\sC
\end{equation}
is fully faithful.  In case $\sC$ is quasi-hereditary, it is an equivalence.
\end{prop}

\begin{prop}\label{prop:tiltdim}
Assume that the tilting and cotilting objects in $\sC$ coincide.  The following conditions are equivalent:
\begin{enumerate}
\item $X \in \fDelta$.
\item There is an exact sequence $0 \to X \to T^0 \to T^1 \to \cdots \to T^k \to 0$
where all the $T^i$ are tilting.\label{it:tilt-resoln}
\end{enumerate}
\end{prop}

Before proving this, we record one immediate consequence.

\begin{defn}\label{defn:tdim}
For $X \in \fDelta$, we define the \emph{tilting dimension} of $X$, denoted $\tdim X$, to be the smallest integer $k$ such that there exists a resolution of $X$ of length $k$ by tilting objects, as in Proposition~\ref{prop:tiltdim}.
\end{defn}

\begin{cor}\label{cor:tdim-induction}
If $X \in \fDelta$, there is a short exact sequence
\begin{equation}\label{eqn:tdim-delta-ses}
0 \to X \to T \to X' \to 0
\end{equation}
where $T$ is tilting, $X' \in \fDelta$, and $\tdim X' = \tdim X - 1$.
\end{cor}

\begin{proof}[Proof of Proposition~\ref{prop:tiltdim}]
Let $\fDelta'$ be the class of objects $X$ satisfying condition~\eqref{it:tilt-resoln} above. The notion of \emph{tilting dimension} makes sense for objects of $\fDelta'$. Moreover, if we replace every occurrence of $\fDelta$ by $\fDelta'$ in the statement of Corollary~\ref{cor:tdim-induction}, then the resulting statement is true.  An argument by induction on tilting dimension, using the short exact sequence~\eqref{eqn:tdim-delta-ses}, shows that $\fDelta' \subset \fDelta$.

Next, let $K^0 \subset \Kb\cT$ be the full subcategory consisting of objects isomorphic to a bounded complex of tilting modules $(X^\bullet, d)$ satisfying the following two conditions:
\begin{enumerate}
\item The complex is concentrated in nonnegative degrees.
\item The cohomology of the complex vanishes, except possibly in degree $0$.
\end{enumerate}
It is easy to see that $\fDelta'$ consists precisely of the objects that lie in the image of $K^0$ under the functor~\eqref{eqn:ringel}. In particular, we see that $\fDelta'$ is stable under extensions, because $K^0$ is.  Thus, to prove that $\fDelta \subset \fDelta'$, it suffices to show that each $\Delta_\gamma$ lies in $\fDelta'$.  This follows from the first short exact sequence in Proposition~\ref{prop:tilt-classif}, by induction on $\gamma$.
\end{proof}

The next lemma is ultimately the source of the torsion-freeness in Theorem~\ref{thm:mv-intro}.

\begin{lem}\label{lem:delta-nabla-free}
Assume that the tilting and cotilting objects in $\sC$ coincide.  If $X \in \fDelta$, then $\uHom(X,\nabla_\gamma)$ is a free module over the graded ring $\uEnd(\nabla_\gamma)$.
\end{lem}
\begin{proof}
We proceed by induction on the number of steps in a standard filtration of $X$.  If $0 \to X' \to X \to X'' \to 0$ is an exact sequence with $X', X'' \in \fDelta$, then we obtain a short exact sequence
\[
0 \to \uHom(X'',\nabla_\gamma) \to \uHom(X,\nabla_\gamma) \to \uHom(X',\nabla_\gamma) \to 0
\]
of $\uEnd(\nabla_\gamma)$-modules.  If the first and last terms are free, the middle term must be as well.  Thus, we are reduced to considering the case where $X$ is a standard object, say $X = \Delta_\xi\la n\ra$.  If $\xi \ne \gamma$, then $\uHom(X, \nabla_\gamma) = 0$.  If $\xi = \gamma$, then $\uHom(X, \nabla_\gamma)$ is a free $\uEnd(\nabla_\gamma)$-module by Lemma~\ref{lem:tilt-cotilt}\eqref{it:delta-nabla-hom}.
\end{proof}

\subsection{Quotients of the category of tilting objects}
\label{subsect:tilting-quot}

The next result compares the Serre quotient $\sC/\sC_\Gamma$ to a ``naive'' quotient category.  If $\sA$ is an additive category and $\sB \subset \sA$ is a full subcategory, we write $\sA \aq \sB$ for the category with the same objects as $\sA$, but with morphisms given by
\begin{equation}\label{eqn:naive-quot}
\Hom_{\sA \aq \sB}(X,Y) = \Hom_\sA(X,Y) / \{f \mid \text{$f$ factors through an object of $\sB$} \}.
\end{equation}

\begin{prop}\label{prop:tilting-quot}
Assume that the tilting and cotilting objects in $\sC$ coincide, and let $\Gamma \subset \Omega$ be a finite order ideal.  The quotient functor $\Pi: \sC \to \sC/\sC_\Gamma$ induces an equivalence of categories
\begin{equation}\label{eqn:tilting-quot}
\bar\Pi: \Tilt(\sC) \aq \Tilt(\sC_\Gamma) \simto \Tilt(\sC / \sC_\Gamma).
\end{equation}
\end{prop}
\begin{proof}
Let $Q: \Tilt(\sC) \to \Tilt(\sC) \aq \Tilt(\sC_\Gamma)$ be the quotient functor.  It is clear that $\Pi(\Tilt(\sC_\Gamma)) = 0$, so there is a unique functor $\bar\Pi$ such that $\bar\Pi \circ Q \cong \Pi$.  From the classification of tilting objects in Proposition~\ref{prop:tilt-classif}, it is clear that every indecomposable object in $\Tilt(\sC / \sC_\Gamma)$ occurs as a direct summand of some object in the image of $\bar\Pi$.  If $\bar\Pi$ were already known to be fully faithful, it would send indecomposable objects to indecomposable objects, and would therefore be essentially surjective.

It suffices, then, to prove that $\bar\Pi$ is fully faithful.  We proceed by induction on the size of $\Gamma$.  Suppose first that $\Gamma$ is a singleton.  Let $T, T' \in \Tilt(\sC)$, and consider the diagram
\begin{equation}\label{eqn:tilting-quot-comp}
\vcenter{\hbox{\begin{tikzpicture}
\node (a) at (0,0) {$\uHom_{\sC}(T,T')$};
\node (b) at (3.75,0) {$\uHom_{\Tilt(\sC) \aq \Tilt(\sC_\Gamma)}(T,T')$};
\node (c) at (8.25,0) {$\uHom_{\sC/\sC_\Gamma}(\Pi(T),\Pi(T'))$};
\draw[->>, >=angle 90, below] (a) -- node {$Q$} (b);
\draw[->, >=angle 90, below] (b) to node {$\bar\Pi$} (c);
\draw[->, >=angle 90,bend left=12, above] (a) to node {$\Pi$} (c);
\end{tikzpicture}}}
\end{equation}
By Lemma~\ref{lem:preserve-Filt}, all three terms of the functorial distinguished triangle $\pil\Pi(T) \to T \to \imath\il(T) \to$ lie in $\sC$, so that distinguished triangle is actually a short exact sequence.  Apply $\uHom({-},T')$ to get the long exact sequence
\begin{multline}\label{eqn:tilting-quot-les}
0 \to \uHom(\il(T),\ir(T')) \to \uHom(T,T') \to \uHom(\Pi(T), \Pi(T')) \to \\
\uExt^1(\il(T),\ir(T'))\to\ldots
\end{multline}
The last term vanishes because (by Lemma~\ref{lem:preserve-Filt} again) $\il(T)$ has a standard filtration and $\ir(T')$ has a costandard filtration.  It follows that the map labelled $\Pi$ in~\eqref{eqn:tilting-quot-comp} is surjective, and its kernel can be identified with the space
\[
K = \{ f: T \to T' \mid \text{$f$ factors as $T \to  \imath\il(T) \to \imath\ir(T') \to T'$} \}.
\]
We deduce that $\bar\Pi$ is surjective as well.  Now, the kernel of $Q$ in~\eqref{eqn:tilting-quot-comp} is the space
\[
K' = \{f: T \to T' \mid \text{$f$ factors through an object of $\Tilt(\sC_\Gamma)$} \}.
\]
We already know that $K' \subset K$.  But since $\Gamma$ is a singleton $\{\gamma\}$ with $\gamma$ necessarily minimal in $\Omega$, we see from Lemma~\ref{lem:tilt-cotilt}\eqref{it:tilt-min} that $\il(T)$ is actually tilting (and not merely in $\fDelta$), and likewise for $\ir(T')$.  So $K = K'$, and we conclude that $\bar\Pi$ in~\eqref{eqn:tilting-quot-comp} is an isomorphism.

For the general case, choose a nonempty proper ideal $\Upsilon \subset \Gamma$.  Then $\Upsilon$ and $\Gamma \setminus \Upsilon$ are both smaller than $\Gamma$, and by induction, we have natural equivalences
\begin{align*}
\Tilt(\sC)\aq \Tilt(\sC_\Upsilon) &\cong \Tilt(\sC/\sC_\Upsilon), \\
\Tilt(\sC_\Gamma)\aq \Tilt(\sC_\Upsilon) &\cong \Tilt(\sC_\Gamma/\sC_\Upsilon), \\
\Tilt(\sC/\sC_\Upsilon) \aq \Tilt(\sC_\Gamma/\sC_\Upsilon) &\cong \Tilt((\sC/\sC_\Upsilon)/(\sC_\Gamma/\sC_\Upsilon)) \cong \Tilt(\sC/\sC_\Gamma).
\end{align*}
It is also easy to see that there is a canonical equivalence
\[
\Tilt(\sC)\aq \Tilt(\sC_\Gamma) \cong (\Tilt(\sC)\aq \Tilt(\sC_\Upsilon)) \bigaq (\Tilt(\sC_\Gamma)\aq \Tilt(\sC_\Upsilon)).
\]
Combining all these yields the desired equivalence~\eqref{eqn:tilting-quot}.
\end{proof}

The next corollary is immediate from~\eqref{eqn:tilting-quot-les} and the discussion following it.

\begin{cor}\label{cor:std-costd-surj}
Assume that the tilting and cotilting objects in $\sC$ coincide, and let $\Gamma \subset \Omega$ be a finite order ideal.  If $X \in \fDelta$ and $Y \in \fnabla$, then the map $\uHom_{\sC}(X,Y) \to \uHom_{\sC/\sC_\Gamma}(\Pi(X), \Pi(Y))$ is surjective.
\end{cor}

\subsection{Perverse-coherent sheaves on the nilpotent cone}
\label{subsect:pervcoh}

In this subsection, we assume that $\bk$ is an algebraically closed field whose characteristic is good for $\Gv$, and that the derived group of $\Gv$ is simply connected.  Recall that $\cN$ denotes the nilpotent cone of $\Gv$.  Let $\Gv \times \Gm$ act on $\cN$ by $(g,z)\cdot x = z^{-2}\Ad(g)(x)$.  We write $\CohGmN$, or simply $\CohN$, for the category of $(\Gv \times \Gm)$-equivariant coherent sheaves on $\cN$.

Let $\PCohN$ denote the category of $(\Gv \times \Gm)$-equivariant \emph{perverse-coherent sheaves} on $\cN$.  This is the heart of a certain remarkable $t$-structure on $\Db\CohN$.  We refer the reader to~\cite{ab:pcs,bez:qes,a} for details on the definition and properties of this category.  Here are some basic facts about $\PCohN$:
\begin{itemize}
\item Every object in $\PCohN$ has finite length.
\item It is stable under $\cF \mapsto \cF\la 1\ra$, where $\la 1\ra: \Db\CohN \to \Db\CohN$ is given by a twist of the $\Gm$-action.
\item The set $\Irr(\PCohN)/\Z$ is naturally in bijection with $\bXp$.
\end{itemize}

\begin{rmk}For any $V\in\Rep(\Gv)$, the coherent sheaf $V\otimes \cON$ is perverse-coherent. The proof of \cite[Lemma 5.4]{a} can be generalized to work for any $V\in\Rep(\Gv)$. Alternatively, one can give a more direct argument using the definition of the perverse-coherent $t$-structure from \cite{bez:qes} and the fact that $\cN$ is Cohen--Macaulay.
\end{rmk}

For $\lambda \in \bXp$, let $\delta_\lambda$ be the length of the shortest Weyl group element $w$ such that $w\lambda$ is antidominant.  We define a subcategory $\PCohN_{\le \lambda} \subset \PCohN$ as in~\eqref{eqn:defn-serre-le}.  (The theorem below implies that this agrees with the definition of $\PCohN_{\le \lambda}$ given in~\S\ref{subsect:proof}.)  For our purposes, the most important fact about $\PCohN$ is the following result of Minn-Thu-Aye, which refines the description given in~\cite{a,bez:qes}.

\begin{thm}[Minn-Thu-Aye~{\cite{myron}}]\label{thm:myron}
Assume that $\chr \bk$ is good for $\Gv$, and that the derived group of $\Gv$ is simply connected.  Then
the category $\PCohN$ is a graded properly stratified category.  Moreover:
\begin{enumerate}
\item The tilting and cotilting objects in $\PCohN$ coincide, and are given by
\[
T_\lambda = (\til(\lambda) \otimes \cON)\la -\delta_\lambda\ra
\qquad\text{and}\qquad
T'_\lambda = (\til(\lambda) \otimes \cON)\la \delta_\lambda\ra.
\]
\item The object $\wey(\lambda) \otimes \cON$ lies in $\PCohN_{\le \lambda}$ and has a standard filtration.
\item The object $\cow(\lambda) \otimes \cON$ lies in $\PCohN_{\le \lambda}$ and has a costandard filtration.
\end{enumerate}
\end{thm}

For completeness, we include a proof of this theorem.  The following argument is adapted from~\cite[Chapter~4]{myron}.

\begin{proof}
Throughout this proof, we will freely make use of the main result of~\cite{a}, which states that $\Db\PCohN$ and $\Db\CohN$ are equivalent.  In particular, we will compute $\Ext$-groups for $\PCohN$ by computing $\Hom$-groups in $\Db\CohN$. 

Let $\tcN$ denote the cotangent bundle of the flag variety for $\Gv$, and let $\pi: \tcN \to \cN$ be the Springer resolution.  Any weight $\lambda \in \bX$ determines a line bundle $\cO_\tcN(\lambda)$ on $\tcN$, obtained by pulling back from the flag variety.  For $\lambda \in \bXp$, let
\[
\bDelta_\lambda = \pi_*\cO_\tcN(w_0\lambda)\la \delta_\lambda\ra
\qquad\text{and}\qquad
\bnabla_\lambda = \pi_*\cO_\tcN(\lambda)\la -\delta_\lambda\ra,
\]
where $w_0$ is the longest element of the Weyl group.  (Here, $\pi_*$ is the derived functor $\Db\Coh(\tcN) \to \Db\CohN$.)  According to~\cite[Proposition~6.1]{a}, parts~\eqref{it:propstd} and~\eqref{it:propcostd} of Definition~\ref{defn:propstrat} hold\footnote{That proposition, like~\cite{bez:qes}, actually asserts that $\PCohN$ is quasi-hereditary, but the papers~\cite{a,bez:qes} use that term in a nonstandard way, imposing weaker $\Ext$-vanishing conditions on standard objects.  Of course, $\PCohN$ is not quasi-hereditary in the sense of this paper.} for $\PCohN$. For $\lambda \in \bXp$, let $L_\lambda$ denote the unique simple subobject of $\bnabla_\lambda$, or equivalently the unique simple quotient of $\bDelta_\lambda$.

That result also says that the $\{\bnabla_\lambda\}$ form a ``graded quasi-exceptional sequence'' (see~\cite[Definition~2.4]{a}). This implies that part~\eqref{it:scalar} of Definition~\ref{defn:propstrat} also holds.  Furthermore, by~\cite[Lemma~4]{bez:qes}, the recollement formalism is available, and hence so are the parts of Lemmas~\ref{lem:preserve-Tilt} and~\ref{lem:preserve-Filt} involving proper standard or proper costandard objects.

Fix $\lambda \in \bXp$, and let $\imath: \Db\PCohN_{<\lambda} \to \Db\PCohN$ and $\Pi: \Db\PCohN \to \Db(\PCohN/\PCohN_{<\lambda}$ be the inclusion and quotient functors, respectively.  We will denote their adjoints as in Proposition~\ref{prop:recollement}.  Let
\[
\Delta_\lambda = \pil\Pi(\wey(\lambda) \otimes \cON)\la -\delta_\lambda\ra.
\]
By~\cite[Lemma~5.4]{a}, $\wey(\lambda) \otimes \cON$ lies in $\PCohN_{\le \lambda}$, so $\Delta_\lambda$ also lies in $\PCohN_{\le \lambda}$ and, by Lemma~\ref{lem:preserve-Filt}, it has a filtration by various $\bDelta_\lambda\la n\ra$.  We claim that
\begin{equation}\label{eqn:delta-head}
\Hom(\Delta_\lambda,\bnabla_\mu\la n\ra)
\cong
\begin{cases}
0 & \text{if $\mu < \lambda$, or if $\mu = \lambda$ and $n \ne 0$,} \\
\bk & \text{if $\mu = \lambda$ and $n = 0$.}
\end{cases}
\end{equation}
By adjunction, we have $\Hom(\Delta_\lambda,\bnabla_\mu\la n\ra) \cong
\Hom(\wey(\lambda) \otimes \cON, \pir\Pi(\bnabla_\mu)\la n+\delta_\lambda\ra)$.  If $\mu < \lambda$, then clearly $\Pi(\bnabla_\mu) = 0$.  If $\mu = \lambda$, then $\pir\Pi(\bnabla_\lambda) \cong \bnabla_\lambda$, and then~\eqref{eqn:delta-head} follows from~\cite[Lemma~5.5]{a}.  

We will next show that for all $\mu \in \bXp$, we have
\begin{equation}\label{eqn:weyl-ext}
\uExt^k(\wey(\lambda) \otimes \cON, \bnabla_\mu) = 0 \qquad\text{for all $k > 0$.}
\end{equation}
By~\cite[Theorem~2]{klt}, the object $\bnabla_\mu$ is actually a coherent sheaf.  Let $\Gamma(\bnabla_\mu)$ be its space of global sections.  As in the proof of~\cite[Lemma~5.5]{a}, we have $\uExt^k(\wey(\lambda) \otimes \cON, \bnabla_\mu) \cong \uExt^k_{\Gv}(\wey(\lambda), \Gamma(\bnabla_\mu))$, and the latter vanishes because, by~\cite[Theorem~7]{klt}, as a $\Gv$-representation, $\Gamma(\bnabla_\mu)$ has a good filtration. 

We also claim that for all $\mu \in \bXp$, we have
\begin{equation}\label{eqn:delta-ext} 
\uExt^k(\Delta_\lambda, \bnabla_\mu) = 0 \qquad\text{for all $k > 0$.}
\end{equation}
If $\mu \ne \lambda$, the claim follows from the recollement formalism.  If $\mu = \lambda$, it follows from~\eqref{eqn:weyl-ext} by a calculation like that used to prove~\eqref{eqn:delta-head} (again invoking~\cite[Lemma~5.5]{a}).

Since every simple object in $\PCohN_{\le\lambda}$ occurs as the socle of some $\bnabla_\mu\la n\ra$ with $\mu \le \lambda$, we see from~\eqref{eqn:delta-head} that $\Delta_\lambda$ has a unique simple quotient, namely $L_\lambda$.
  
Next, let $K_\mu$ denote the cokernel of $L_\mu \hookrightarrow \bnabla_\mu$, and consider the exact sequence
\[
\cdots \to \Hom(\Delta_\lambda, K_\mu\la n\ra) \to \Ext^1(\Delta_\lambda, L_\mu\la n\ra) \to \Ext^1(\Delta_\lambda, \bnabla_\mu\la n\ra) \to \cdots.
\]
If $\mu \le \lambda$, then $K_\mu$ must lie in $\PCohN_{<\lambda}$, and the preceding paragraph implies that the first term vanishes.  The last term vanishes by~\eqref{eqn:delta-ext}, so the middle term vanishes for all $\mu \le \lambda$.  Thus, $\Delta_\lambda$ is a projective object in $\PCohN_{\le \lambda}$, and hence the projective cover of $L_\lambda$.  Since $L_\lambda$ is the unique simple quotient of $\bDelta_\lambda$, $\Delta_\lambda$ is also the projective cover of $\bDelta_\lambda$.  We have now established part~\eqref{it:standard} of Definition~\ref{defn:propstrat}.  The first half of part~\eqref{it:ext2} holds by~\eqref{eqn:delta-ext}.

Let $\mathcal{S}$ be the Serre--Grothendieck duality functor on $\Db\CohN$.  This is an antiautoequivalence that preserves $\PCohN$ and swaps $\bDelta_\lambda$ with $\bnabla_{-w_0\lambda}$.  Define
\[
\nabla_\lambda = \mathcal{S}(\Delta_{-w_0\lambda}).
\]
(In fact, one can check that $\nabla_\lambda \cong \pir\Pi(\cow(\lambda) \otimes \cON)\la \delta_\lambda\ra$.)  It follows from the previously established properties of $\Delta_\lambda$ that part~\eqref{it:costandard} and the second half of part~\eqref{it:ext2} of Definition~\ref{defn:propstrat} hold for $\PCohN$.  We have completed the proof that $\PCohN$ is a graded properly stratified category.

We saw earlier that $\wey(\lambda) \otimes \cON$ lies in $\PCohN_{\le \lambda}$.  By~\eqref{eqn:weyl-ext} and the criterion in~\cite[Theorem~1.6(iii)]{ahlu}, we see that $\wey(\lambda) \otimes \cON$ has a standard filtration. By Serre--Grothendieck duality, $\cow(\lambda) \otimes \cON$ lies in $\PCohN_{\le \lambda}$ and has a costandard filtration.  It follows immediately that $\til(\lambda) \otimes \cON$ lies in $\PCohN_{\le \lambda}$ and is both tilting and cotilting.  There are nonzero maps $\Delta_\lambda \to \wey(\lambda) \otimes \cON\la-\delta_\lambda\ra \to \til(\lambda) \otimes \cON\la-\delta_\lambda\ra$ and $\til(\lambda) \otimes \cON\la \delta_\lambda\ra \to \cow(\lambda) \otimes \cON\la \delta_\lambda\ra \to \nabla_\lambda$.  We thus obtain the desired formulas for $T_\lambda$ and $T'_\lambda$.
\end{proof}

Note that this theorem does \emph{not} say that $\wey(\lambda) \otimes \cON$ is itself a standard object.  Indeed, the standard objects in $\PCohN$ do not, in general, belong to $\CohN$.  The costandard objects of $\PCohN$ do happen to lie in $\CohN$, but they are not generally of the form $\cow(\lambda) \otimes \cON$.

\begin{cor}\label{cor:pcoh-even}
Let $\Gamma \subset \bXp$ be a finite order ideal.  Suppose $V_1 \in \Rep(\Gv)$ has a Weyl filtration, and $V_2 \in \Rep(\Gv)$ has a good filtration.  Then the graded vector space $\uHom(\Pi(V_1 \otimes \cON), \Pi(V_2 \otimes \cON))$ is concentrated in even degrees.
\end{cor}
\begin{proof}
When $\Gamma = \varnothing$, it is clear that the space $\uHom_{\CohN}(V_1 \otimes \cON, V_2 \otimes \cON) \cong \uHom_{\Rep(\Gv)}(V_1, V_2 \otimes \bk[\cN])$ is concentrated in even degrees, since the coordinate ring $\bk[\cN]$ is concentrated in even degrees.  For general $\Gamma$, the result then follows from Corollary~\ref{cor:std-costd-surj}.
\end{proof}

\section{Background on Parity sheaves}
\label{sect:parity}

Let $X$ be a complex algebraic variety or ind-variety equipped with a fixed algebraic stratification (as in \cite[Definition 3.2.23]{cg}) $X = \bigsqcup_{\gamma \in \Omega} X_\gamma$, where $\Omega$ is some indexing set.  In the ind-variety case, we assume that the closure of each $X_\gamma$ is an ordinary finite-dimensional variety; in particular, the closure of each stratum should contain only finitely many other strata.  Let $\bk$ be a field.  Assume the following conditions hold:
\begin{itemize}
\item Each stratum $X_\gamma$ is simply connected.
\item The cohomology groups $H^k(X_\gamma;\bk)$ vanish when $k$ is odd.
\end{itemize}
Let $\Db(X,\bk)$, or simply $\Db(X)$, denote the bounded derived category of $\bk$-sheaves on $X$ (in the analytic topology).  Let $\Dbom(X,\bk)$, or simply $\Dbom(X)$, denote the full triangulated subcategory consisting of complexes that are supported on the union of finitely many strata and are constructible with respect to the given stratification.  For each stratum $X_\gamma$, let $j_\gamma: X_\gamma \to X$ be the inclusion map. For a locally closed subspace $Y\subset X$, we denote the constant sheaf on $Y$ by~$\ubk_Y$. 

\begin{defn}\label{defn:parity}
An object $\cF \in \Dbom(X)$ is said to be \emph{$*$-even} (resp.~\emph{$!$-even}) if for each $\gamma$, the cohomology sheaves $\cH^k(j_\gamma^*\cF)$ (resp.~$\cH^k(j_\gamma^!\cF)$) vanish for $k$ odd.  It is \emph{even} if it is both $*$-even and $!$-even.

The terms \emph{$*$-odd}, \emph{$!$-odd}, and \emph{odd} are defined similarly.  An object is \emph{parity} if it is a direct sum of an even 
object and an odd object.
\end{defn}

The assumptions above are significantly more restrictive than those in~\cite{jmw}, but we will not require the full generality of~{\it loc.~cit}.  The following statement classifies the indecomposable parity objects.

\begin{thm}[{\cite[Theorem~2.12]{jmw}}]
Let $\cE$ be an indecomposable parity object.  Then there is a stratum $X_\gamma$ such that $\cE$ is supported on $\overline{X_\gamma}$, and $\cE|_{X_\gamma}$ is a shift of the constant sheaf $\ubk_{X_\gamma}$.  Moreover, if $\cE'$ is another indecomposable parity object with the same support as $\cE$, then $\cE'$ is (up to shift) isomorphic to $\cE$.
\end{thm}

\begin{defn}
The variety $X$ is said to \emph{have enough parity objects} if for every stratum $X_\gamma$, there is an indecomposable parity object $\cE_\gamma$ that is supported on the closure $\overline{X_\gamma}$, and such that $\cE_\gamma|_{X_\gamma} \cong \ubk_{X_\gamma}[\dim X_\gamma]$.
\end{defn}

For $X$ as above, let $\Parity(X) \subset \Dbom(X)$ denote the full additive subcategory consisting of parity objects.  The main result of this section is the following geometric analogue of Proposition~\ref{prop:tilting-quot}, comparing a Verdier quotient of $\Dbom(X)$ to a ``naive'' quotient (cf.~\eqref{eqn:naive-quot}).  The statement makes use of the following observation: for any closed inclusion  of a union of strata $i: Y \to X$, we can identify $\Parity(Y)$ with a full subcategory of $\Parity(X)$ via $i_*$.  

\begin{prop}\label{prop:parity-quot}
Assume that $X$ has enough parity objects, and let $Y \subset X$ be a closed union of finitely many strata.  The open inclusion $j: X \setminus Y \to X$ induces an equivalence of categories
\begin{equation}\label{eqn:parity-quot}
\bar{\jmath^*}: \Parity(X) \aq \Parity(Y) \simto \Parity(X \setminus Y).
\end{equation}
\end{prop}
\begin{proof}
Let $Q: \Parity(X) \to \Parity(X) \aq \Parity(Y)$ be the quotient functor, and let $i: Y \to X$ be the inclusion map.  It is clear that $j^*(\Parity(Y)) = 0$, so there is unique functor $\bar{\jmath^*}$ such that $\bar{\jmath^*} \circ Q \cong j^*$.  Because $X$ has enough parity objects, every indecomposable parity object occurs as a direct summand of some object in the image of $\bar{\jmath^*}$.  By~\cite[Proposition~2.11]{jmw}, $\bar{\jmath^*}$ sends indecomposable objects to indecomposable objects, so it is essentially surjective.

It remains to prove that $\bar{\jmath^*}$ is fully faithful.  We proceed by induction on the number of strata in $Y$.  Suppose first that $Y$ consists of a single closed stratum $X_0$.  Let $\cE,\cF \in \Dbom(X)$ be parity objects, and consider the diagram
\begin{equation}\label{eqn:parity-quot-comp}
\vcenter{\hbox{\begin{tikzpicture}
\node (a) at (0,0) {$\Hom(\cE,\cF)$};
\node (b) at (4,0) {$\Hom_{\Parity(X)\aq \Parity(Y)}(\cE,\cF)$};
\node (c) at (8,0) {$\Hom(j^*\cE,j^*\cF)$};
\draw[->>, >=angle 90, below] (a) -- node {$Q$} (b);
\draw[->, >=angle 90, below] (b) to node {$\bar{\jmath^*}$} (c);
\draw[->, >=angle 90,bend left=12, above] (a) to node {$j^*$} (c);
\end{tikzpicture}}}
\end{equation}
It suffices to consider the case where $\cE$ and $\cF$ are both indecomposable.  If $\cE$ is even and $\cF$ is odd, or vice versa, then both $\Hom(\cE,\cF)$ and $\Hom(j^*\cE,j^*\cF)$ vanish by~\cite[Corollary~2.8]{jmw}, so $\bar{\jmath^*}$ is trivially an isomorphism.  We henceforth assume that $\cE$ and $\cF$ are both even.  (The case where they are both odd is identical.)  Apply $\Hom({-},\cF)$ to the distinguished triangle $j_!j^*\cE \to \cE \to i_*i^*\cE \to$ to get the long exact sequence
\[
\cdots \to \Hom(i^*\cE,i^!\cF) \to
\Hom(\cE,\cF) \overset{j^*}{\to} \Hom(j^*\cE,j^*\cF) \to \Hom^1(i^*\cE,i^!\cF) \to \cdots
\]
Since $i^*\cE$ is $*$-even and $i^!\cF$ is $!$-even, we see from~\cite[Corollary~2.8]{jmw} that the last term above vanishes.  It follows that the map labelled $j^*$ in~\eqref{eqn:parity-quot-comp} is surjective, and its kernel can be identified with the space
\[
K = \{ f: \cE \to \cF \mid \text{$f$ factors as $\cE \to i_*i^*\cE \to i_*i^!\cF \to \cF$ } \}.
\]
We deduce that $\bar{\jmath^*}$ is surjective as well.  Now, the kernel of $Q$ in~\eqref{eqn:parity-quot-comp} is the space
\[
K' = \{ f : \cE \to \cF \mid \text{$f$ factors through an object of $\Parity(Y)$ } \}.
\]
We already know that $K' \subset K$.  But since $Y$ consists of a single closed stratum, the object $i^*\cE$ is actually even (not just $*$-even), and likewise for $i^!\cF$.  So $K = K'$, and we conclude that $\bar{\jmath^*}$ in~\eqref{eqn:parity-quot-comp} is an isomorphism.

For the general case, let $S$ be an open stratum in $Y$, and let $B = Y \setminus S$ and $X' = X \setminus B$.  Then $S$ is closed in $X'$, and by induction, we have natural equivalences
\begin{align*}
\Parity(X') \aq \Parity(S) &\simto \Parity(X' \setminus S) \cong \Parity(X \setminus Y), \\
\Parity(X) \aq \Parity(B) &\simto \Parity(X \setminus B) \cong \Parity(X'), \\
\Parity(Y) \aq \Parity(B) &\simto \Parity(Y \setminus B) \cong \Parity(S).
\end{align*}
The desired equivalence~\eqref{eqn:parity-quot} follows from these and the general observation that
\[
\Parity(X)\aq \Parity(Y) \cong (\Parity(X)\aq \Parity(B)) \bigaq (\Parity(Y)\aq \Parity(B)).\qedhere
\]
\end{proof}

\section{Parity sheaves on Kac--Moody flag varieties}
\label{sect:ginzburg}

In this section, we study Ext-groups of parity sheaves on flag varieties for Kac--Moody groups.  The result below
will be applied elsewhere in the paper only to affine Grassmannians, but it is no more effort to prove it in this generality.  
As in the previous section, $\bk$ denotes an arbitrary field.

\begin{thm}\label{thm:ginzburg}
Let $X$ be a generalized flag variety for a Kac--Moody group, equipped with the Bruhat stratification, and let $\cE$ and $\cF$  be two parity objects with respect to that stratification.  The natural map
\[
\Hom^\bullet_{\Db(X)}(\cE, \cF) \to \Hom^\bullet_{H^\bullet(X;\bk)}(H^\bullet(\cE), H^\bullet(\cF))
\]
is an isomorphism.
\end{thm}

For finite flag varieties, this result (with some minor restrictions on $\chr \bk$) is due to Soergel~\cite{soergel1, soergel2}.  In~\cite{gin:pscsa}, Ginzburg proved a very similar result for simple perverse $\C$-sheaves 
on smooth projective varieties equipped with a suitable $\C^\times$-action.  The proof below follows the outline of Ginzburg's argument quite closely.  One exception occurs at a step (see~\cite[Proposition~3.2]{gin:pscsa}) where Ginzburg invokes the theory of mixed Hodge modules: here, we substitute an argument of Fiebig--Williamson that relies on the geometry of Schubert varieties.

\begin{rmk}
When $\bk = \C$, Ginzburg had already observed in a remark at the end of~\cite{gin:pscsa} that his result could be generalized to the Kac--Moody case.  Thus, in that case, this section can be regarded as an exposition of Ginzburg's remark.
\end{rmk}

\begin{rmk}
Although only field coefficients are used in the present paper, the arguments in this section go through unchanged if we instead take $\bk$ to be a complete discrete valuation ring, so Theorem~\ref{thm:ginzburg} holds in that setting.
\end{rmk}

We begin with some notation.  Let $\cG$ be a Kac--Moody group (over $\C$), with maximal torus $\cT \subset \cG$ and standard Borel subgroup $\cB \subset \cG$.  Let $\cB^- \subset \cG$ denote the opposite Borel subgroup to $\cB$ (with respect to $\cT$).  Let $\cP \subset \cG$ be a standard parabolic subgroup of finite type (in the sense of \cite[\S1.2.2]{kum:kmg}), with Levi factor~$L_\cP$.

For the remainder of the section, $X$ will denote the generalized flag variety $X = \cG/\cP$.  Let $W$ (resp.~$W_\cP$) be the Weyl group of $\cG$ (resp.~$L_\cP$), and let $W^\cP$ be the set of minimal-length representatives of the set of cosets $W/W_\cP$.  The length of an element $w \in W^\cP$ will be denoted by $\ell(w)$.  It is well known that the $\cT$-fixed points and the $\cB$-orbits on $X$ are both naturally in bijection with $W^\cP$.  For $w \in W^\cP$, let $e_w \in X$ be the corresponding $\cT$-fixed point, and let $X_w = \cB \cdot e_w$ be the corresponding Bruhat cell.  We will also need the ``opposite Bruhat cell'' $X^-_w = \cB^- \cdot e_w$.  Recall that $X_w \cap X^-_w = \{e_w\}$, and that the intersection is transverse.  Moreover, $X_w$ is isomorphic to an affine space of dimension $\ell(w)$.  In general, $X^-_w$ may have infinite dimension, but it has codimension $\ell(w)$ (see~\cite[Lemma~7.3.10]{kum:kmg}).  

Recall that $\Db(X)$ is the bounded derived category of all $\bk$-sheaves on $X$. Let $\Dbb(X)$ be the full subcategory of $\Db(X)$ consisting of complexes $\mathcal{F}$ such that
\begin{enumerate}
\item $\mathcal{F}$ is constructible with respect to the stratification by $\cB$-orbits, and 
\item the support of $\cF$ is contained in the union of finitely many $\bar X_w$.
\end{enumerate}
Let $\ParityB(X) \subset \Dbb(X)$ be the additive category of parity objects. Let
\[
j_w: X_w \to X
\qquad\text{and}\qquad
\bar\jmath_w: X^-_w \to X
\]
denote the inclusion maps. For any closed subset $Z \subset X$, we let $i_Z: Z \to X$ be the inclusion map, and for an object $\cE \in \Dbb(X)$, we put
\[
\cE_Z := i_{Z*}i_Z^*\cE
\qquad\text{and}\qquad
\cE^Z := i_{Z*}i_Z^!\cE.
\]
For simplicity, the inclusions of closures of $\cB$- and $\cB^-$-orbits are denoted
\[
i_w: \bar X_w \to X
\qquad\text{and}\qquad
\bar\imath_w: \bar X^-_w \to X,
\]
rather than $i_{\bar X_w}$ and $i_{\bar X^-_w}$.  Recall that $\bar X_w$ is known as a \emph{Schubert variety}, and $\bar X^-_w$ as a \emph{Birkhoff variety}.

Certain complexes $\mathcal{F}\in\Db(X)$ that do \emph{not} belong to $\Dbb(X)$ will also appear in our arguments, including objects whose support may be infinite-dimensional and hence not locally compact.  Some caution is required when working with such objects, especially when applying functors of the form $f_!$ and $f^!$.  In this section, whenever infinite-dimensional supports are involved, the functors $f_!$ and $f^!$ will only be used when $f$ is an inclusion of a locally closed subset.  See Appendix~\ref{app:nonloccpt} for the definition and basic properties of these functors on non-locally compact spaces.

\begin{lem}
Let $Z \subset X$ be a finite union of Schubert varieties, and let $X_w \subset Z$ be a Bruhat cell that is open in $Z$.  If $\cE$ is $*$-even, then for each $k$, there is a natural short exact sequence
\[
0 \to H^k(j_{w!} j_w^! \cE_Z) \to H^k(\cE_Z) \to H^k(\cE_{Z \setminus X_w}) \to 0.
\]
\end{lem}
\begin{proof}
The constant map $a: Z \to \pt$ is a proper, even morphism in the sense of~\cite[Definition~2.33]{jmw}, so by~\cite[Proposition~2.34]{jmw}, if $\cE' \in \Dbb(Z)$ is $*$-even, then $H^k(\cE')$ vanishes when $k$ is odd. All three terms in the distinguished triangle $j_{w!}j_w^! \cE_Z \to \cE_Z \to \cE_{Z \setminus X_w}\to$ are $*$-even, so in the long exact sequence in cohomology, all odd terms vanish, and the even terms give short exact sequences as above.
\end{proof}
 
\begin{lem}
Let $Z \subset X$ be a finite union of Schubert varieties, and let $X_w \subset Z$ be a Bruhat cell that is open in $Z$.  If $\cE$ is a parity object, then the natural map $H^k(\cE_Z) \to H^k(j_w^*\cE_Z)$ is surjective.
\end{lem}
\begin{proof}
Let $k_w: \{e_w\} \to X$ be the inclusion map.  Since $k_w$ factors through $j_w$, and $j_w$ factors through $i_Z$, there is a natural sequence of maps
\[
\cE \to i_{Z*}i_Z^*\cE \to j_{w*}j_w^*\cE \to k_{w*}k_w^*\cE.
\]
Taking cohomology, we obtain a natural sequence of maps
\begin{equation}\label{eqn:par-stalk-surj}
H^k(\cE) \to H^k(\cE_Z) \to H^k(j_w^*\cE) \to H^k(k_w^*\cE).
\end{equation}
We claim first that the composition $H^k(\cE) \to H^k(k_w^*\cE)$ is surjective.  This is essentially the content of~\cite[Theorem~5.7(2)]{fw:psmg}.  That result is stated in an abstract, axiomatic setting, but~\cite[Proposition~7.1]{fw:psmg} tells us that it applies to Schubert varieties.  Another concern is that~\cite[Theorem~5.7(2)]{fw:psmg} deals with $\cT$-equivariant rather than ordinary cohomology.  The reader may check that the proof goes through with ordinary cohomology as well.  Alternatively, note that both $H^\bullet_\cT(\cE)$ and $H^\bullet_\cT(k_w^*\cE)$ are free modules over the equivariant cohomology ring of a point $H^\bullet_\cT(\pt)$ by~\cite[Proposition~5.6]{fw:psmg}.  In that situation, the ordinary cohomology is obtained from the equivariant cohomology by applying the right-exact functor $\bk \otimes_{H^\bullet_{\cT}(\pt)} {-}$.  In particular, the surjectivity of $H^\bullet_\cT(\cE) \to H^\bullet_\cT(k_w^*\cE)$ implies the surjectivity of the corresponding map in ordinary cohomology.

Next, we claim that the third map in the sequence~\eqref{eqn:par-stalk-surj} is an isomorphism.  Since $j_w^*\cE$ lies in the triangulated subcategory of $\Db(X_w)$ generated by the constant sheaf $\ubk_{X_w}$, it suffices to check that $H^k(\ubk_{X_w}) \to H^k(\ubk_{\{e_w\}})$ is an isomorphism.  That last claim is obvious.

From these observations, it follows that $H^k(\cE_Z) \to H^k(j_w^*\cE) \cong H^k(j_w^*\cE_Z)$ is surjective as well.
\end{proof}

\begin{lem}\label{lem:reldual-opp}
There is a canonical isomorphism $\bar \jmath_w^!\ubk_X \cong \ubk_{X^-_w}[-2\ell(w)]$.
\end{lem}
\begin{proof}
Let $\cU \subset \cB$ be the pro-unipotent radical of the Borel subgroup.  For $w \in W^\cP$, let $\cU_w \subset \cU$ be the subgroup generated by the root subgroups $\cU_\alpha$ where $\alpha$ is a positive root but $w^{-1}\alpha$ is negative.  Then $\cU_w$ is a finite-dimensional unipotent algebraic group.  Let $O_w = \cU_w \cdot X^-_w \subset X$.  According to~\cite[Lemma~7.3.10]{kum:kmg}, the multiplication map
\[
\cU_w \times X^-_w \to O_w
\]
is an isomorphism of ind-varieties.  Let $i: X^-_w \to \cU_w \times X^-_w$ be the inclusion map $x \mapsto (e,x)$, where $e \in \cU_w$ is the identity element.  Since $O_w$ is open in $X$, $\bar\jmath_w^!\ubk_X$ is naturally isomorphic to $i^!\ubk_{\cU_w \times X^-_w}$.  

To compute $i^!\ubk_{\cU_w \times X^-_w}$, let us consider the Cartesian square
\[
\begin{tikzpicture}[description/.style={fill=white,inner sep=1.5pt}] 
\matrix (m) [matrix of math nodes, row sep=2.5em,
column sep=2.5em, text height=1ex, text depth=0.25ex, nodes in empty cells]
{ X^-_w & \cU_w \times X^-_w \\
  \{e\} & \cU_w \\};
\path[->,>=angle 90, font=\scriptsize]
(m-1-1) edge node[above]{$i$} (m-1-2)
edge node[left]{$a$} (m-2-1)
(m-1-2) edge node[right]{$q$} (m-2-2)
(m-2-1) edge node[above]{$i_0$} (m-2-2);
\end{tikzpicture}
\]
where $i_0$ is the inclusion map, and $a$ and $q$ are the obvious projection maps.  The space $X^-_w$ is contractible, and each cohomology sheaf $\cH^k(i^!\ubk_{\cU_w \times X^-_w})$ is $\cB^-$-equivariant, so in fact each such cohomology sheaf must be a constant sheaf.  It follows from, say,~\cite[Corollary~2.7.7]{ks} that the adjunction map
\[
a^*a_*i^!\ubk_{\cU_w \times X^-_w} \to i^!\ubk_{\cU_w \times X^-_w}
\]
is an isomorphism.  Thus, to finish the proof, we must show that $a_*i^!\ubk_{\cU_w \times X^-_w} \cong \ubk_{\{e\}}[-2\ell(w)]$.  Using Lemma~\ref{lem:basechange} and~\cite[Corollary~2.7.7]{ks} again, we find that
\[
a_*i^!\ubk_{\cU_w \times X^-_w} \cong i_0^!q_*\ubk_{\cU_w \times X^-_w}
\cong i_0^!q_*q^* \ubk_{\cU_w} \cong i_0^!\ubk_{\cU_w}.
\]
Since $\cU_w$ is isomorphic as a variety to an affine space $\mathbb{A}^{\ell(w)}$, we have a well-known canonical isomorphism $i_0^!\ubk_{\cU_w} \cong \ubk_{\{e\}}[-2\ell(w)]$, and the result follows.
\end{proof}

Now, let $Y \subset X$ be a finite union of Birkhoff varieties, and let
\[
\Lambda_Y = i_Y^!\ubk_X.
\]

\begin{lem}\label{lem:reldual-bound}
Let $d_Y = \min \{ \ell(w) \mid X^-_w \subset Y\}$.  Then the cohomology sheaves $\cH^i(\Lambda_Y)$ vanish for $i < 2d_Y$.
\end{lem}
\begin{proof}
Let $X^-_w \subset Y$ be such that $\ell(w) = d_Y$.  Then $X^-_w$ is necessarily open in $Y$.  Let $A_1 = X^-_w$, and then inductively define $A_2, A_3, \ldots$ by setting $A_k$ to be some opposite Bruhat cell $X^-_y$ that is open in $Y \setminus (A_1 \cup \cdots \cup A_{k-1})$.  We also let $B_k = A_1 \cup \cdots \cup A_k$.  The $B_k$'s form an increasing sequence of open subsets of $Y$ whose union is all of $Y$.  To show that $\cH^i(\Lambda_Y)$ vanishes for $i < 2d_Y$, it suffices to show that for all $k$,
\begin{equation}\label{eqn:reldual}
\cH^i(\Lambda_Y|_{B_k}) = 0 \qquad\text{for $i < 2d_Y$.}
\end{equation}
We proceed by induction on $k$.  For $k = 1$, we have $\Lambda_Y|_{B_1} \cong \bar\jmath_w^!\ubk_X$, so~\eqref{eqn:reldual} follows from Lemma~\ref{lem:reldual-opp}.  For $k > 1$, let $a: A_k \to B_k$ and $b: B_{k-1} \to B_k$ be the inclusion maps.  We have a distinguished triangle
\[
a_*a^!(\Lambda_Y|_{B_k}) \to \Lambda_Y|_{B_k} \to b_*b^*(\Lambda|_{B_k}) \to.
\]
If $A_k = X^-_y$, the first term can be identified with $a_*\bar\jmath_y^!\ubk_X$.  Since $\ell(y) \ge d_Y$, it follows from Lemma~\ref{lem:reldual-opp} again that $\cH^i(a_*a^!(\Lambda_Y|_{B_k})) = 0$ for $i < 2d_Y$.  On the other hand, $b^*(\Lambda|_{B_k}) \cong \Lambda|_{B_{k-1}}$, so $\cH^i(b_*b^*(\Lambda_Y|_{B_k}))$ vanishes for $i < 2d_Y$ by induction.  Thus,~\eqref{eqn:reldual} holds, as desired.
\end{proof}

\begin{lem}
For any $w \in W^\cP$, there is a canonical morphism $q_w: \ubk_{\bar X^-_w} \to \Lambda_{\bar X^-_w}[2\ell(w)]$ whose restriction to $X^-_w \subset \bar X^-_w$ is an isomorphism. 
\end{lem}
\begin{proof}
Let $Y = \bar X^-_w \setminus X^-_w$, and let $y: Y \to \bar X^-_w$ and $u: X^-_w \to \bar X^-_w$ be the inclusion maps.  The space $Y$ is a finite union of Birkhoff varieties, and, moreover, the integer $d_Y$ defined in Lemma~\ref{lem:reldual-bound} satisfies $d_Y > \ell(w)$.  Consider the distinguished triangle
\[
y_!\Lambda_Y \to \Lambda_{\bar X^-_w} \to u_*u^*\Lambda_{\bar X^-_w} \to.
\]
Note that $u^*\Lambda_{\bar X^-_w} \cong \bar\jmath_w^! \ubk_X$.  It follows from Lemma~\ref{lem:reldual-opp} by adjunction that we have a canonical morphism
\begin{equation}\label{eqn:reldual-can}
\ubk_{\bar X^-_w} \to u_*u^*\Lambda_{\bar X^-_w}[2\ell(w)].
\end{equation}
By construction, this map restricts to an isomorphism over $X^-_w$.
Next, we deduce from Lemma~\ref{lem:reldual-bound} that
\[
\Hom(\ubk_{\bar X^-_w}, y_!\Lambda_Y[2 \ell(w)]) = \Hom(\ubk_{\bar X^-_w}, y_!\Lambda_Y[2\ell(w)+1]) = 0.
\]
These facts imply that the map in~\eqref{eqn:reldual-can} factors in a unique way through $\Lambda_{\bar X^-_w}$.  The resulting map $\ubk_{\bar X^-_w} \to \Lambda_{\bar X^-_w}[2\ell(w)]$ is the one we seek.
\end{proof}
\begin{rmk} The map $q_w: \ubk_{\bar X^-_w} \to \Lambda_{\bar X^-_w}[2\ell(w)]$ plays a role similar to that of the fundamental class in Borel--Moore homology.
\end{rmk}

Define $c_w: \ubk_X \to \ubk_X[2\ell(w)]$ to be the composition
\begin{equation}\label{eqn:cw-defn}
\ubk_X \to \bar\imath_{w*} \ubk_{\bar X^-_w} \xrightarrow{\bar\imath_{w*}q_w} \bar\imath_{w*}\Lambda_{\bar X^-_w}[2\ell(w)] \to \ubk_X[2\ell(w)].
\end{equation}
We now study the map $\bar c_w: \Hom^k(\ubk_X, \cE_Z) \to \Hom^{k+2\ell(w)}(\ubk_X,\cE_Z)$ induced by $c_w$.

\begin{lem}\label{lem:ginz-comm-E}
Let $Z$ be a finite union of Schubert varieties, let $X_w \subset Z$ be a Bruhat cell that is open in $Z$, and let $\cE \in \ParityB(X)$. Then $\bar c_w$ induces an isomorphism $H^k(j_w^*\cE_Z) \simto H^{k+2\ell(w)}(j_{w!}j_w^!\cE_Z)$ that makes the following diagram commute:
\begin{center}
\begin{tikzpicture}[description/.style={fill=white,inner sep=1.5pt}] 

\matrix (m) [matrix of math nodes, row sep=2.5em,
column sep=1em, text height=1ex, text depth=0.25ex, nodes in empty cells]
{ & & H^k(\cE_Z) & H^k(j_w^*\cE_Z) & \\
  0 & H^{k+2\ell(w)}(\cE_{Z \setminus X_w}) & H^{k+2\ell(w)}(\cE_Z) & H^{k + 2\ell(w)}(j_{w!}j_w^!\cE_Z) & 0\\ };
\path[->, >=angle 90, font=\scriptsize]

(m-1-3) edge node[right] {$\bar c_w$} (m-2-3)
(m-1-4) edge node[auto] {$\bar c_w$} node[sloped, below] {$\sim$} (m-2-4)
(m-2-2) edge (m-2-1)
(m-2-5) edge (m-2-4);

\path[->>, >=angle 90, font=\scriptsize]
(m-1-3) edge (m-1-4)
(m-2-3) edge (m-2-2);

\path[left hook->, >=angle 90, font=\scriptsize]
(m-2-4) edge (m-2-3);

\end{tikzpicture}
\end{center}
\end{lem}
\begin{proof}
Let $Y = Z \setminus X_w$, and let $y: Y \to X$ be the inclusion map.  Since $c_w$ factors through an object supported on $\bar X^-_w$, and since $\bar X^-_w \cap Y = \varnothing$, we see that the composition $H^k(y^!\cE_Z) \to H^k(\cE_Z) \overset{\bar c_w}{\to} H^{k+2\ell(w)}(\cE_Z)$ vanishes.  In other words, $\bar c_w$ must factor through $H^k(\cE_Z) \to H^k(j_w^*\cE_Z)$.  The resulting map $H^k(j_w^*\cE_Z) \to H^{k+2\ell(w)}(\cE_Z)$ must factor through $H^{k+2\ell(w)}(j_{w!}j_w^!\cE_Z) \to H^{k+2\ell(w)}(\cE_Z)$ for the same reason, so we at least have a commutative diagram as shown above.

It remains to show that the right-hand vertical map is an isomorphism.  Let $p: \{e_w\} \to X_w$ be the inclusion map.  Applying $i_w^*$ to~\eqref{eqn:cw-defn} yields the composition $\ubk_{\bar X_w} \to j_{w*}p_*\ubk_{e_w} \to \ubk_{\bar X_w}[2\ell(w)]$, where the second map comes from adjunction and the identification $\ubk_{e_w} \cong p^!\ubk_{X_w}[2\ell(w)]$.  Note that $j_{w!}p_*\ubk_{e_w} \cong j_{w*}p_*\ubk_{e_w}$.  Thus, $\bar c_w$ is given by the following composition:
\begin{multline*}
H^k(j_w^*\cE_Z) = \Hom^k(\ubk_{\bar X_w}, j_{w*}j_w^*\cE_Z) \to \Hom^k(j_{w*}p_*\ubk_{e_w}[-2\ell(w)], j_{w*}j_w^*\cE_Z) \\
\cong \Hom^{k+2\ell(w)}(p_*\ubk_{e_w}, j_w^!\cE_Z)
\cong \Hom^{k+2\ell(w)}(j_{w!}p_*\ubk_{e_w}, j_{w!}j_w^!\cE_Z) \\
\to 
\Hom^{k+2\ell(w)}(\ubk_{\bar X_w}, j_{w!}j_w^!\cE_Z) = H^{k+2\ell(w)}(j_{w!}j_w^!\cE_Z).
\end{multline*}
We claim that each step of this is an isomorphism.  If it happens that $j_w^*\cE_Z \cong \ubk_{X_w}[m]$, this can be checked explicitly.  But because $\cE_Z$ is $*$-parity, $j_w^*\cE_Z$ is always isomorphic to a direct sum of various $\ubk_{X_w}[m]$.
\end{proof}

A very similar argument establishes the following result, whose proof we omit.

\begin{lem}\label{lem:ginz-comm-F}
Let $Z$ be a finite union of Schubert varieties, let $X_w \subset Z$ be a Bruhat cell that is open in $Z$, and let $\cF \in \ParityB(X)$. Then $\bar c_w$ induces an isomorphism $H^k(j_w^*\cF_Z) \simto H^{k+2\ell(w)}(j_{w!}j_w^!\cF^Z)$ that makes the following diagram commute:
\begin{center}
 \begin{tikzpicture}[description/.style={fill=white,inner sep=1.5pt}] 

\matrix (m) [matrix of math nodes, row sep=2.5em,
column sep=1em, text height=1ex, text depth=0.25ex, nodes in empty cells]
{ 0 & H^{k}(\cF^{Z \setminus X_w}) & H^{k}(\cF^Z) & H^{k}(j_{w}^*\cF^Z) & 0\\
& & H^{k+2\ell (w)}(\cF^Z) & H^{k+2\ell (w)}(j_{w!}j_w^!\cF^Z) & \\};
\path[->, >=angle 90,  font=\scriptsize]

(m-1-3) edge node[right] {$\bar c_w$} (m-2-3)
(m-1-4) edge node[auto] {$\bar c_w$}  node[below, sloped] {$\sim$} (m-2-4)
(m-1-1) edge (m-1-2)
(m-1-4) edge (m-1-5);

\path[->>, >=angle 90, font=\scriptsize]
(m-1-3) edge (m-1-4);

\path[right hook->, >=angle 90, font=\scriptsize]
(m-1-2) edge (m-1-3);

\path[left hook->, >=angle 90, font=\scriptsize]
(m-2-4) edge (m-2-3);

\end{tikzpicture}\end{center}
\end{lem}

With the following proposition, we complete the proof of Theorem~\ref{thm:ginzburg}.

\begin{prop}\label{prop:ginzburg}
Let $Z \subset X$ be a finite union of Schubert varieties, and let $\cE, \cF \in \ParityB(X)$.  The natural map
\[
\Hom^\bullet_{\Dbb(X)}(\cE_Z, \cF^Z) \to \Hom^\bullet_{H^\bullet(X;\bk)}(H^\bullet(\cE_Z), H^\bullet(\cF^Z))
\]
is an isomorphism.
\end{prop}
\begin{proof}[Proof Sketch]
This is proved by induction on the number of Bruhat cells in $Z$, via a diagram chase relying on formal consequences of the commutative diagrams in Lemmas~\ref{lem:ginz-comm-E} and~\ref{lem:ginz-comm-F}.  The argument is essentially identical to the proof of~\cite[Proposition~3.10]{gin:pscsa}; see also~\cite[Eq.~(3.8a--b)]{gin:pscsa}.  We omit further details. 
\end{proof}

\section{Ext-groups of parity sheaves on the affine Grassmannian}
\label{sect:extgr}

In this section, $\bk$ will denote an algebraically closed field whose characteristic is good for $\Gv$.  We also assume that $\Gv$ satisfies~\eqref{eqn:reasonable}.  Recall that the $\Go$-orbits are parametrized by $\bXp$.  If $\Gamma \subset \bXp$ is a finite order ideal, we can form the closed subset $\Gr_\Gamma = \bigcup_{\gamma \in \Gamma} \Gr_\gamma$. Let
\begin{equation}\label{eqn:openinclu}
j_\Gamma: U_\Gamma = \Gr \setminus \Gr_\Gamma \hookrightarrow \Gr
\end{equation}
be the complementary open inclusion.  For the remainder of the paper, all constructible complexes on $\Gr$ or on any subset of $\Gr$ should be understood to be constructible with respect to the $\Go$-orbits.  In particular, $\ParityGo(\Gr)$ will denote the category of $\Go$-constructible parity objects.

Our goal is compute certain $\Ext$-groups in $\Dbgo(\Gr)$ or in some $\Dbgo(U_\Gamma)$ in terms of $\PCohN$.
The main result, a modular generalization of~\cite[Prop\-o\-si\-tion~1.10.4]{gin:pslg}, depends on a result of Yun--Zhu~\cite{yz} describing the cohomology of $\Gr$.  We begin by recalling that result.

Let $\fe \in \Lie(\Gv)$ be the principal nilpotent element described in~\cite[Proposition~5.6]{yz}.  Let $\Bv \subset \Gv$ be the unique Borel subgroup such that $\fe \in \Lie(\Bv)$, and let $\Uv \subset \Bv$ be its unipotent radical. Recall that $\Gv \times \Gm$ acts on $\cN$ by $(g,z)\cdot x = z^{-2}\Ad(g)(x)$. Below, for any subgroup $H \subset \Gv \times \Gm$, we write $H_\fe$ for the stabilizer of $\fe$ in $H$. 

Let $\Gr^\circ$ be the identity component of $\Gr$, and let $\Dist(\Uv_\fe)$ denote the algebra of distributions on $\Uv_\fe$ with support at the identity (see~\cite[\S I.7.1 and~\S I.7.7]{jan:rag}).

\begin{thm}[Yun--Zhu~\cite{yz}]\label{thm:yz}
There is a natural isomorphism
\begin{equation}\label{eqn:yz}
H^\bullet(\Gr^\circ;\bk) \cong \Dist(\Uv_\fe).
\end{equation}
\end{thm}

The ``naturality'' in this proposition refers to a certain compatibility with $\Sat$.  To be more precise, given $M \in \PervGo$, the isomorphism~\eqref{eqn:yz} endows $H^\bullet(M)$ with the structure of a $\Dist(\Uv_\fe)$-module.  On the other hand, if we forget the grading on $H^\bullet(M)$, we obtain the underlying vector space of $\Sat^{-1}(M) \in \Rep(G^\vee)$.  Thus, we can regard $H^\bullet(M)$ as a representation of $\Uv_\fe \subset G^\vee$, and hence as a $\Dist(\Uv_\fe)$-module.  In fact, these two $\Dist(\Uv_\fe)$-module structures on $H^\bullet(M)$ coincide.

\begin{rmk}
Theorem~\ref{thm:yz} is stated in~\cite{yz} only when $G$ is quasi-simple and simply connected (in which case $\Gr = \Gr^\circ$), but it is easily extended to general $G$ by routine arguments.  One caveat is that the element $\fe$ may depend on a choice in general (it is uniquely determined in the quasi-simple case).  Once $\fe$ is fixed, however, the isomorphism~\eqref{eqn:yz} is still natural in the sense described above.
\end{rmk}

Let $\Gm$ act on $\Gv$ by conjugation via the cocharacter $2\rho: \Gm \to \Tv$, where $2\rho$ is the sum of the positive roots for $G$.  The resulting semidirect product will be denoted $\Gm \ltimes_{2\rho} \Gv$.  This action preserves the subgroups $\Gv_\fe$, $\Bv_\fe$, and $\Uv_\fe$, so groups such as $\Gm \ltimes_{2\rho} \Uv_\fe$ also make sense.

\begin{lem}\label{lem:fe-centralizer}
There are isomorphisms $(\Gv \times \Gm)_\fe \cong \Gm \ltimes_{2\rho} \Bv_\fe \cong \Gm \ltimes_{2\rho} \Uv_\fe \times \Cent(\Gv)$, where $\Cent(\Gv)$ denotes the center of $\Gv$.
\end{lem}
\begin{proof}
Let $\phi: \Gv \times \Gm \to \Gm \ltimes_{2\rho} \Gv$ be the map $\phi(g,z) = (z,2\rho(z^{-1})g)$.  This map is an isomorphism, and it is easily checked that it takes $(\Gv \times \Gm)_\fe$ to $\Gm \ltimes_{2\rho} \Gv_\fe$.  By~\cite[Theorem~5.9(b)]{springer}, $\Gv_\fe = \Bv_\fe \cong \Uv_\fe \times \Cent(\Gv)$.
\end{proof}

\begin{lem}\label{lem:nilp-dist}
Let $M_1,M_2 \in \Perv_{\Go}(\Gr^\circ,\bk)$.  There is a natural isomorphism
\[
\uHom_{\CohN}(\Sat^{-1}(M_1) \otimes \cON, \Sat^{-1}(M_2) \otimes \cON) \simto \uHom_{\Dist(\Uv_\fe)}(H^\bullet(M_1), H^\bullet(M_2)).
\]
\end{lem}
In the course of the proof of this lemma, we will encounter an analogous statement (see~\eqref{eqn:nilp-free-rep} below) that is entirely in terms of $\Gv$-modules, and that does not involve the geometric Satake equivalence.
\begin{proof}
The assumptions on $\bk$ imply that $\cN$ is a normal variety (see, e.g.,~\cite[Prop\-o\-si\-tion~8.5]{jan:nort}).  Let $\cN_\reg \subset \cN$ be the subvariety consisting of regular nilpotent elements, and let $\cO_\reg$ denote its structure sheaf.  Let $h: \cN_\reg \hookrightarrow \cN$ be the inclusion map, and let $V \in \Rep(\Gv)$.  Since the complement of $\cN_\reg$ has codimension at least $2$, the restriction map $h^*: \Gamma(V \otimes \cON) \to \Gamma(V \otimes \cO_\reg)$ is an isomorphism of $(\Gv \times \Gm)$-modules.  Equivalently, the adjunction map $V \otimes \cON \to R^0h_*h^*(V \otimes \cON)$ is an isomorphism.  It follows that $h^*$ is fully faithful on the category of free sheaves in $\CohN$.  In particular, if we set
\begin{equation}\label{eqn:nilp-dist-satake}
V_1 = \Sat^{-1}(M_1) = H^\bullet(M_1)
\qquad\text{and}\qquad
V_2 = \Sat^{-1}(M_2) = H^\bullet(M_2),
\end{equation}
then $h^*$ gives us a natural isomorphism
\[
\uHom_{\CohN}(V_1 \otimes \cON, V_2 \otimes \cON) \simto \uHom_{\CohGm(\cN_\reg)}(V_1 \otimes \cO_\reg, V_2 \otimes \cO_\reg).
\]
Now, $\cN_\reg$ is the orbit of the point $\fe$ under $\Gv$ or $\Gv \times \Gm$. Thanks to condition~\eqref{eqn:reasonable}, the natural map $\Gv/\Gv_\fe \to \cN_\reg$ is an isomorphism of varieties.  (See~\cite[\S2.9]{jan:nort}, for example.)  Factoring this map as $\Gv/\Gv_\fe \to (\Gv \times \Gm)/(\Gv \times \Gm)_\fe \to \cN_\reg$, we see that $(\Gv \times \Gm)/(\Gv \times \Gm)_\fe \to \cN_\reg$ is also an isomorphism.  Therefore, restriction to $\fe$ induces an equivalence of categories $\CohGm(\cN_\reg) \simto \Rep((\Gv \times \Gm)_\fe)$.  In view of Lemma~\ref{lem:fe-centralizer}, we have a natural isomorphism
\begin{equation}\label{eqn:nilp-free-rep}
\uHom_{\CohN}(V_1 \otimes \cON, V_2 \otimes \cON) \simto \uHom_{\Gm \ltimes_{2\rho} \Uv_\fe \times \Cent(\Gv)}(V_1, V_2).
\end{equation}
Since the $M_i$ are supported on $\Gr^\circ$, $\Cent(\Gv)$ acts trivially on the $V_i$, so we may simply omit mentioning it and consider $\Hom_{\Gm \ltimes_{2\rho} \Uv_\fe}(V_1, V_2)$.

Now, the category of finite-dimensional $\Uv_\fe$-representations can be identified with a full subcategory of the finite-dimensional $\Dist(\Uv_\fe)$-modules~\cite[Lemma I.7.16]{jan:rag}.  Similarly, the category of finite-dimensional $(\Gm \ltimes_{2\rho} \Uv_\fe)$-modules can be identified with a full subcategory of \emph{graded} finite-dimensional $\Dist(\Uv_\fe)$-modules, where $\Dist(\Uv_\fe)$ itself is graded by the action of $\Gm$ via the cocharacter $2\rho$.  This is precisely the grading appearing in~\eqref{eqn:yz}, according to the remarks following~\cite[Theorem~1.1]{yz}.  On the other hand, the grading on the right-hand side of each equation in~\eqref{eqn:nilp-dist-satake} is also given by $2\rho$, as seen in~\cite[Theorem~3.6]{mv}.  Thus, we have
\[
\uHom_{\Gm \ltimes_{2\rho} \Uv_\fe}(V_1, V_2) \cong \uHom_{\Dist(\Uv_\fe)}(H^\bullet(M_1), H^\bullet(M_2)),
\]
and the result follows.
\end{proof}

\begin{prop}\label{prop:hom-tilt-map}
For all $V_1, V_2 \in \Rep(\Gv)$, there is a natural map
\begin{equation}\label{eqn:hom-tilt}
\Evsat: \Hom^i_{\DGo}(\Sat(V_1), \Sat(V_2)) \to \Hom_{\CohN}(V_1 \otimes \cON, V_2 \otimes \cON\la i \ra).
\end{equation}
When $\Sat(V_1)$ and $\Sat(V_2)$ are parity sheaves, this is an isomorphism.  For any $V_1$ and $V_2$, this map is compatible with composition; i.e., the following diagram commutes:
\begin{equation}\label{eqn:hom-tilt-map-compose} \hspace{-2em}
\begin{tikzpicture}[description/.style={fill=white,inner sep=1.5pt}, baseline=(current  bounding  box.center),] 
\tikzstyle{every node}=[font=\small]
\matrix (m) [matrix of math nodes, row sep=2.5em,
column sep=1em, text height=1ex, text depth=0.25ex, nodes in empty cells]
{\Hom^\bullet(\Sat(V_2), \Sat(V_3)) \otimes \Hom^\bullet(\Sat(V_1),\Sat(V_2)) & \Hom^\bullet(\Sat(V_1), \Sat(V_3))  \\
 \uHom(V_2 \otimes \cON, V_3 \otimes \cON) \otimes \uHom(V_1 \otimes \cON, V_2 \otimes \cON)&  \uHom(V_1 \otimes \cON, V_3 \otimes \cON)\\ };
\path[->,>=angle 90, font=\scriptsize]
(m-1-1) edge (m-1-2)
edge (m-2-1)
(m-1-2) edge (m-2-2)
(m-2-1) edge (m-2-2);
\end{tikzpicture}
\end{equation}
\end{prop}
Note that the naturality of~\eqref{eqn:hom-tilt} just means that it is compatible with composition of morphisms in $\Perv_\Go(\Gr)$.  The diagram~\eqref{eqn:hom-tilt-map-compose} expresses a stronger property, allowing arbitrary morphisms in $\Dbgo(\Gr)$.
\begin{proof}
We construct the map~\eqref{eqn:hom-tilt} as the following composition:
\begin{multline}\label{eqn:hom-tilt-cohom}
\Hom^\bullet(\Sat(V_1), \Sat(V_2)) \xrightarrow{H^\bullet}
\uHom_{H^\bullet(\Gr)} (H^\bullet(\Sat(V_1)), H^\bullet(\Sat(V_2))) \\
\xrightarrow[\sim]{\text{Lemma~\ref{lem:nilp-dist}}}
\uHom(V_1 \otimes \cON, V_2 \otimes \cON)
\end{multline}
When $\Sat(V_1)$ and $\Sat(V_2)$ are parity sheaves, Theorem~\ref{thm:ginzburg} tells us that the first map in~\eqref{eqn:hom-tilt-cohom} is an isomorphism.  Note that the first map in~\eqref{eqn:hom-tilt-cohom} is induced by a functor defined on all of $\Dbgo(\Gr)$, while the second map is essentially the inverse of~\eqref{eqn:nilp-free-rep}, which is induced by a functor defined on all of $\CohN$. These two observations imply the commutativity of~\eqref{eqn:hom-tilt-map-compose}.
\end{proof}

For the remainder of this section, we will assume that $\chr \bk$ is also a JMW prime (Definition~\ref{defn:jmw}) for $\Gv$.  Recall that most good primes are known to be JMW:

\begin{thm}[{\cite[Theorem~1.8]{jmw:parity}}]
Assume $\Gv$ is quasi-simple.  If $\chr \bk$ satisfies the bounds in Table~\ref{tab:jmw}, then $\Sat$ sends every tilting module to a parity sheaf.
\end{thm}

\begin{prop}\label{prop:evsat-exist}
There is an equivalence of additive categories
\[
\Evsat: \ParityGo(\Gr) \to \Tilt(\PCohN)
\]
such that for any tilting $\Gv$-module $V$, we have $\Evsat(\Sat(V)[n]) \cong V \otimes \cON\la n\ra$.
\end{prop}
\begin{proof}
Every indecomposable object in $\ParityGo(\Gr)$ is isomorphic to an object of the form $\Sat(V)[n]$, where $V \in \Rep(\Gv)$ is a tilting module.  Similarly, every indecomposable  tilting object in $\PCohN$ is of the form $V \otimes \cON\la n\ra$ for such a $V$.  Thus, Proposition~\ref{prop:hom-tilt-map} implies that the full subcategory of indecomposable objects in $\ParityGo(\Gr)$ is equivalent to the full subcategory of indecomposable objects in $\Tilt(\PCohN)$.  Such an equivalence extends in a unique way (up to isomorphism) to an equivalence $\ParityGo(\Gr) \simto \Tilt(\PCohN)$.
\end{proof}

\begin{cor}\label{cor:compare-quotient}
Let $\Gamma \subset \bXp$ be a finite order ideal.  There is an equivalence of categories $\Evsat_\Gamma$, unique up to isomorphism, that makes the following diagram commute up to isomorphism:

\begin{center}
\begin{tikzpicture}[description/.style={fill=white,inner sep=1.5pt}] 

\matrix (m) [matrix of math nodes, row sep=2.5em,
column sep=1em, text height=1ex, text depth=0.25ex, nodes in empty cells]
{ \ParityGo(\Gr)  & \Tilt(\PCohN) \\
  \ParityGo(U_\Gamma) & \Tilt(\PCohN/\PCohN_\Gamma)\\ };
\path[->,>=angle 90, font=\scriptsize]
(m-1-1) edge node[above]{$\Evsat$} (m-1-2)
edge node[left]{$j_\Gamma^*$} (m-2-1)
(m-1-2) edge node[left]{$\Pi_\Gamma$} (m-2-2)
(m-2-1) edge node[above]{$\Evsat_\Gamma$} (m-2-2);
\end{tikzpicture}
\end{center}
\end{cor}
\begin{proof}
The functor $\Evsat$ of Proposition~\ref{prop:evsat-exist} restricts to an equivalence of categories $\ParityGo(\Gr_\Gamma) \simto \Tilt(\PCohN_\Gamma)$, and so it induces an equivalence 
\[
\ParityGo(\Gr) \aq \ParityGo(\Gr_\Gamma) \simto \Tilt(\PCohN) \aq \Tilt(\PCohN_\Gamma).
\]
Propositions~\ref{prop:tilting-quot} and~\ref{prop:parity-quot} then give us the result.
\end{proof}

\begin{thm}\label{thm:hom-tilt}
Let $\Gamma \subset \bXp$ be a finite order ideal.  If $V_1$ has a Weyl filtration and $V_2$ a good filtration, there is a natural isomorphism  of graded vector spaces
\begin{multline}
\Pvsat_\Gamma: \Hom^\bullet_{\Dbgo(U_\Gamma)}(S(V_1)|_{U_\Gamma}, S(V_2)|_{U_\Gamma}) \simto  \label{eqn:hom-tilt-open}\\
\uHom_{\PCohN/\PCohN_\Gamma}(\Pi_\Gamma(V_1 \otimes \cON), \Pi_\Gamma(V_2 \otimes \cON)).
\end{multline}
This map is compatible with~\eqref{eqn:hom-tilt}, in the sense that the  diagram
\begin{equation}\label{eqn:hom-tilt-compare}
\hspace{-2.2em}
\begin{tikzpicture}[description/.style={fill=white,inner sep=1.5pt}, baseline=(current  bounding  box.center)]
\tikzstyle{every node}=[font=\small]
\matrix (m) [matrix of math nodes, row sep=2.5em,
column sep=1.5em, text height=1.5ex, text depth=0.25ex, nodes in empty cells]
{ \Hom^\bullet_{\Dbgo(\Gr)}(\Sat(V_1), \Sat(V_2)) &  \uHom_{\PCohN}(V_1 \otimes \cON, V_2 \otimes \cON)  \\
  \Hom^\bullet_{\Dbgo(U_\Gamma)}(\Sat(V_1)|_{U_\Gamma}, \Sat(V_2)|_{U_\Gamma}) &  \uHom_{\frac{\PCohN}{\PCohN_\Gamma}}(\Pi_\Gamma(V_1 \otimes \cON), \Pi_\Gamma(V_2 \otimes \cON))\\ };
\path[->, >=angle 90, font=\scriptsize]
(m-1-1) edge node[above]{$\Evsat$} node[below]{$\sim$} (m-1-2)
edge node[left]{$j_\Gamma^*$} (m-2-1)
(m-1-2) edge node[right]{$\Pi_\Gamma$} (m-2-2)
(m-2-1) edge node[above]{$\Pvsat_\Gamma$}  node[below]{$\sim$} (m-2-2);
\end{tikzpicture}
\end{equation}
commutes.  Moreover,~\eqref{eqn:hom-tilt-open} is compatible with composition: if $V_1$ has a Weyl filtration, $V_2$ is tilting, and $V_3$ has a good filtration, then the following diagram commutes:

\begin{equation}\label{eqn:hom-tilt-compose}
\hspace{-1.2em}
\begin{tikzpicture}[description/.style={fill=white,inner sep=1.5pt}, baseline=(current  bounding  box.center)]
\tikzstyle{every node}=[font=\small]
\matrix (m) [matrix of math nodes, row sep=3em,
column sep=1em, text height=1.5ex, text depth=.25ex, nodes in empty cells]
{ {\begin{array}{l}\Hom^\bullet(\Sat(V_2)|_{U_\Gamma}, \Sat(V_3)|_{U_\Gamma}) \otimes {}\\ \qquad \Hom^\bullet(\Sat(V_1)|_{U_\Gamma},\Sat(V_2)|_{U_\Gamma})\end{array}} &  \Hom^\bullet(\Sat(V_1)|_{U_\Gamma}, \Sat(V_3)|_{U_\Gamma})   \\
  {\begin{array}{l}\uHom(\Pi_\Gamma(V_2 \otimes \cON), \Pi_\Gamma(V_3 \otimes \cON)) \otimes {}\\ \qquad\uHom(\Pi_\Gamma(V_1 \otimes \cON), \Pi_\Gamma(V_2 \otimes \cON))\end{array}} &  \uHom(\Pi_\Gamma(V_1 \otimes \cON), \Pi_\Gamma(V_3 \otimes \cON))\\ };
\path[->, >=angle 90, font=\scriptsize]
(m-1-1) edge (m-1-2)
(m-1-2) edge (m-2-2)
(m-2-1) edge (m-2-2);
\path[->, >=angle 90, font=\scriptsize, shorten >=0.15cm,shorten <=.15cm]
(m-1-1) edge (m-2-1);

\end{tikzpicture}
\end{equation}
\end{thm}

Note that, in contrast with~\eqref{eqn:hom-tilt}, the map $\Pvsat_\Gamma$ is only defined when $V_1$ has a Weyl filtration and $V_2$ a good filtration, and not for general objects of $\Rep(\Gv)$.

\begin{proof}
For brevity, we will write $\Pvsat$ for $\Pvsat_\Gamma$, $U$ for $U_\Gamma$, and $\Pi$ for $\Pi_\Gamma$ throughout the proof.  We proceed by induction on the tilting dimension of $V_1$ and $V_2$.

If $V_1$ and $V_2$ both have tilting dimension $0$, i.e., if they are both tilting, then we simply take $\Pvsat$ to be the map induced by the functor $\Evsat_\Gamma$ from Corollary~\ref{cor:compare-quotient}.  The commutativity of both~\eqref{eqn:hom-tilt-compare} and~\eqref{eqn:hom-tilt-compose} is immediate from that proposition.

Suppose now that the result is known when $V_1$ has tilting dimension${}\le n_1$ and $V_2$ has tilting dimension${}\le n_2$.  Now, let $V_1 \in \Rep(G^\vee)$ have a Weyl filtration and tilting dimension $n_1+1$.  By Corollary~\ref{cor:tdim-induction}, we can find a short exact sequence
\begin{equation}\label{eqn:hom-tilt-induction}
0 \to V_1 \to T \to V_1' \to 0
\end{equation}
where $T$ is tilting and $V_1'$ has a Weyl filtration and tilting dimension $n_1$.

Let $V_2$ have a good filtration and tilting dimension${}\le n_2$.  Let $i$ be an even integer, and consider the commutative diagram in Figure~\ref{fig:hom-tilt-diagram}.  The left-hand column is the long exact sequence induced by~\eqref{eqn:hom-tilt-induction}.  The right-hand column is also induced by~\eqref{eqn:hom-tilt-induction}.  It is a short exact sequence because, by Theorem~\ref{thm:myron} and Lemma~\ref{lem:preserve-Tilt}, $\Pi(V_1' \otimes \cON)$ and $\Pi(V_2 \otimes \cON)$ have standard and costandard filtrations, respectively, in $\PCohN/\PCohN_{\Gamma}$, and so $\uExt^1(\Pi(V_1' \otimes \cON), \Pi(V_2 \otimes \cON)) = 0$.
\begin{figure}
\begin{tikzpicture}[description/.style={fill=white,inner sep=1.5pt}, baseline=(current  bounding  box.center)] 
 \tikzstyle{every node}=[font=\small]

\matrix (m) [matrix of math nodes, row sep=2.5em,
column sep=1em, text height=1ex, text depth=0.25ex]
{ \Hom^{i-1}(\Sat(T)|_U,\Sat(V_2)|_U)=0 &   \\
  \Hom^{i-1}(\Sat(V_1)|_U,\Sat(V_2)|_U) & 0\\
 \Hom^i(\Sat(V_1')|_U,\Sat(V_2)|_U) & \Hom(\Pi(V_1' \otimes \cON), \Pi(V_2 \otimes \cON)\la i\ra)\\ 
 \Hom^i(\Sat(T)|_U,\Sat(V_2)|_U) & \Hom(\Pi(T \otimes \cON), \Pi(V_2 \otimes \cON)\la i \ra)\\
\Hom^i(\Sat(V_1)|_U,\Sat(V_2)|_U) & \Hom(\Pi(V_1 \otimes \cON), \Pi(V_2 \otimes \cON)\la i \ra) \\
\Hom^{i+1}(\Sat(V_1')|_U,\Sat(V_2)|_U) & 0\\ };
\path[->,>=angle 90, font=\scriptsize]
(m-1-1) edge  (m-2-1)
(m-2-1) edge (m-3-1)
(m-3-1) edge node[left]{$u$} (m-4-1)
edge node[above]{$\Pvsat$} node[below]{$\sim$} (m-3-2)
(m-4-1) edge node[left]{$v$} (m-5-1)
edge node[above]{$\Pvsat$} node[below]{$\sim$}(m-4-2)
(m-5-1) edge (m-6-1)
(m-6-1) edge node[above]{$\Pvsat$} node[below]{$\sim$}  (m-6-2)
(m-2-2) edge  (m-3-2)
(m-3-2) edge node[left]{$u'$} (m-4-2)
(m-4-2) edge node[left]{$v'$} (m-5-2)
(m-5-2) edge (m-6-2);

\path[->,>=angle 90, font=\scriptsize, color=white]
(m-3-1) edge node[color=black]{$(*)$} (m-2-4);

\path[dashed, ->,>=angle 90, font=\scriptsize]
(m-5-1) edge (m-5-2);
\end{tikzpicture}
\caption{Commutative diagram for the proof of Theorem~\ref{thm:hom-tilt}}\label{fig:hom-tilt-diagram}
\end{figure}

By induction, we have $\Hom^\bullet(\Sat(T)|_U, \Sat(V_2)|_U) \cong \uHom(\Pi(T \otimes \cON), \Pi(V_2 \otimes \cON))$.  In particular, since $i$ is even, we have $\Hom^{i-1}(\Sat(T)|_U, \Sat(V_2)|_U) = 0$ by Corollary~\ref{cor:pcoh-even}.  Next, because the square $(*)$
involving $u$ and $u'$
commutes, the map $u$ must be injective.  It follows that
\[
\Hom^{i-1}(\Sat(V_1)|_U, \Sat(V_2)|_U) = 0 \qquad\text{for $i-1$ odd.}
\]
The left-hand column of Figure~\ref{fig:hom-tilt-diagram} has now been reduced to a short exact sequence.  It is clear that there is a unique isomorphism
\begin{equation}\label{eqn:hom-tilt-even}
\Pvsat: \Hom^i(\Sat(V_1)|_U,\Sat(V_2)|_U) \simto \Hom(\Pi(V_1 \otimes \cON), \Pi(V_2 \otimes \cON)\la i \ra)
\end{equation}
that makes Figure~\ref{fig:hom-tilt-diagram} commute.  For now, the map we have constructed appears to depend on the choice of~\eqref{eqn:hom-tilt-induction}.  We will address this issue later.

First, let us consider the special case where $\Gamma = \varnothing$, so that $U = \Gr$, and $\Pi$ is the identity functor.  In this case, the solid horizontal arrows in Figure~\ref{fig:hom-tilt-diagram} are given by~\eqref{eqn:hom-tilt}, by induction.  Since the dotted arrow is uniquely determined, it too must be given by~\eqref{eqn:hom-tilt}.  In particular, we have now shown that the top horizontal arrow in~\eqref{eqn:hom-tilt-compare} is an isomorphism for the pair $(V_1,V_2)$.

Now, compare the special case ($\Gamma = \varnothing$) of Figure~\ref{fig:hom-tilt-diagram} with the general case.  Since~\eqref{eqn:hom-tilt-compare} holds for the pairs $(V_1', V_2)$ and $(T,V_2)$ by induction, one can see by an easy diagram chase that it also holds for the pair $(V_1, V_2)$.

Recall from Corollary~\ref{cor:std-costd-surj} that the right-hand vertical map in~\eqref{eqn:hom-tilt-compare} is surjective.  Since the horizontal maps are isomorphisms, the left-hand vertical map must be surjective as well.  Once we know that both vertical maps are surjective, we can see that the bottom horizontal map is uniquely determined.  Thus, the map~\eqref{eqn:hom-tilt-even} is independent of~\eqref{eqn:hom-tilt-induction}.

It remains to show that~\eqref{eqn:hom-tilt-even} is natural in both variables, and that~\eqref{eqn:hom-tilt-compose} commutes.  The former is essentially a special case of the latter, so we focus on the latter.  In the special case $\Gamma = \varnothing$, the commutativity of~\eqref{eqn:hom-tilt-compose} is contained in Proposition~\ref{prop:hom-tilt-map}.  For general $\Gamma$, we deduce the result by a diagram chase using the special case $\Gamma = \varnothing$ together with several instances of~\eqref{eqn:hom-tilt-compare}.

An entirely similar argument establishes the required induction step involving the tilting dimension of $V_2$.
\end{proof}

\section{Proof of the Mirkovi\'c--Vilonen conjecture}
\label{sect:mv}

In this section, $\bk$ may be any field.
We begin with a lemma about sheaves on a single $\Go$-orbit $\Gr_\lambda$ in $\Gr$.
Note that $\Dbgo(\Gr_\lambda)$ is the category of complexes of sheaves whose cohomology sheaves are (locally) constant.  

\begin{lem}\label{lem:even-free}
The following conditions on an object $\cF \in \Dbgo(\Gr_\lambda)$ are equivalent:
\begin{enumerate}
\item $\cF$ is even.
\item $\Hom^\bullet(\cF,\ubk_{\Gr_\lambda})$ is a free $H^\bullet(\Gr_\lambda)$-module generated in even degrees.
\end{enumerate}
\end{lem}
\begin{proof}
Every even object is a direct sum of objects of the form $\ubk_{\Gr_\lambda}[2n]$, so it is clear that the first condition implies the second.  Suppose now that the second condition holds.  Choose a basis $e_1, \ldots, e_k$ for $\Hom^\bullet(\cF,\ubk_{\Gr_\lambda})$ as a $H^\bullet(\Gr_\lambda)$-module, and suppose each $e_i$ is homogeneous of degree $2n_i$.  That is, each $e_i$ is a morphism $\cF \to \ubk_{\Gr_\lambda}[2n_i]$.  Let $\cF' = \bigoplus_{i=1}^k \ubk_{\Gr_\lambda}[2n_i]$, and consider the map
$
f = (e_1, \ldots, e_k): \cF \to \cF'.
$
The map $\Hom^\bullet(\cF',\ubk_{\Gr_\lambda}) \to \Hom^\bullet(\cF,\ubk_{\Gr_\lambda})$ induced by $f$ is surjective (all the $e_i$ lie in its image), and since these are finite-dimensional graded vector spaces, it is an isomorphism.  Therefore, letting $\cG$ denote the cone of $f: \cF \to \cF'$, we have
\begin{equation}\label{eqn:even-free-cone}
\Hom^\bullet(\cG,\ubk_{\Gr_\lambda}) = 0.
\end{equation}
We claim that $\cG = 0$.  If not, let $n$ be the top degree in which $\mathcal{H}^n(\cG) \ne 0$. Then, there is a nonzero truncation morphism $\cG \to \tau_{\ge n}\cG \cong \mathcal{H}^n(\cG)[-n]$.  The constant sheaf $\mathcal{H}^n(\cG)$ is a direct sum of copies of $\ubk_{\Gr_\lambda}$, so there is a nonzero map $\cG \to \ubk_{\Gr_\lambda}[-n]$, contradicting~\eqref{eqn:even-free-cone}.  Thus, $\cG = 0$, and $f$ is an isomorphism.  In particular, $\cF \cong \bigoplus \ubk_{\Gr_\lambda}[2n_i]$ is even.
\end{proof}

\begin{thm}[{cf.~\cite[Conjecture~1.10]{jmw:parity}}]\label{thm:jmwconj}
Assume that $\chr \bk$ is a JMW prime for $\Gv$.  Then the perverse sheaf $\cIstd(\lambda,\bk)$ is $*$-parity.
\end{thm}
More precisely, $\cIstd(\lambda,\bk)$ is $*$-even (resp.~$*$-odd) if $\dim \Gr_\lambda$ is even (resp.~odd).
\begin{proof}
Let $\bar\bk$ be an algebraic closure of $\bk$.  For any $x\in\Gr$, we have that $\cIstd(\lambda,\bar\bk)_x\cong\cIstd(\lambda,\bk)_x \otimes \bar\bk$.  Thus, if $\cIstd(\lambda,\bar\bk)$ were known to be $*$-parity, the result would follow for $\cIstd(\lambda,\bk)$.  In other words, it suffices to prove the theorem for algebraically closed fields.  For the remainder of the proof, we assume that $\bk$ is algebraically closed.

It is well known that every component of $\Gr$ is isomorphic to a component of the affine Grassmannian for the group $G/\Cent(G)$, via an isomorphism compatible with the stratification by $\Go$-orbits.  For groups of type $A$, there is a similar statement in the opposite direction: every component of the affine Grassmannian of $\mathrm{PGL}(n)$ is isomorphic to some component of the affine Grassmannian of $\mathrm{GL}(n)$.  These two observations imply that to prove the theorem in general, it suffices to consider groups of the form
\begin{equation}\label{eqn:g-reduce}
\mathrm{GL}(n_1) \times \cdots \times \mathrm{GL}(n_k) \times \left(
\begin{array}{c}
\text{a semisimple group of adjoint type}\\
\text{containing no factors of type $A$}
\end{array}
\right)
\end{equation}
Assume henceforth that $G$ has this form.  Then $\Gv$ is a product of $\mathrm{GL}(n_i)$'s with a semisimple, simply-connected group containing no factors of type $A$.  According to~\cite[\S2.9]{jan:nort}, such a group satisfies~\eqref{eqn:reasonable}, so we can invoke the results of~\S\ref{sect:extgr}.

For simplicity, let us assume that $\dim \Gr_\lambda$ is even; the argument in the odd case is the same.  Let $\mu \in \bXp$ be a weight such that $\Gr_\mu \subset \overline{\Gr_\lambda}$.  Recall that this implies that $\dim \Gr_\mu$ is also even.

Let $\Gamma \subset \bXp$ be the set of weights that are strictly smaller than $\mu$.  Let $U = U_\Gamma = \Gr \setminus \Gr_\Gamma$, and let $j=j_\Gamma$ as in \eqref{eqn:openinclu}.  By Theorem~\ref{thm:hom-tilt}, we have a natural isomorphism
\[
\Hom^\bullet(j^*\cIstd(\lambda,\bk),j^*\cT(\mu,\bk)) \cong \uHom(\Pi(\wey(\lambda) \otimes \cON), \Pi(\til(\mu) \otimes \cON)).
\]
In particular, by~\eqref{eqn:hom-tilt-compose}, this is an isomorphism of graded modules over the ring
\[
\Hom^\bullet(j^*\cT(\mu,\bk),j^*\cT(\mu,\bk)) \cong \uEnd(\Pi(\til(\mu) \otimes \cON)).
\]
Finally, in the quotient category $\PCohN/\PCohN_\Gamma$, the tilting object $\Pi(\til(\mu) \otimes \cON)$ coincides with the costandard object $\Pi(\cow(\mu) \otimes \cON)$, by Lemma~\ref{lem:tilt-cotilt}\eqref{it:tilt-min}.  By Lemma~\ref{lem:delta-nabla-free}, the space $\uHom(\Pi(\wey(\lambda) \otimes \cON), \Pi(\cow(\mu) \otimes \cON))$ is a free $\uEnd(\Pi(\cow(\mu) \otimes \cON))$-module, and by Corollary~\ref{cor:pcoh-even}, it is generated in even degrees.  

Now, let $i: \Gr_\mu \to U$ be the inclusion map.  This is a closed inclusion, and we clearly have $j^*\cT(\mu,\bk) \cong i_*\ubk_{\Gr_\mu}[\dim \Gr_\mu]$.  Rephrasing the conclusion of the previous paragraph, we have that
\[
\Hom^\bullet(j^*\cIstd(\lambda,\bk),i_*\ubk_{\Gr_\mu}) \cong
\Hom^\bullet(i^*j^*\cIstd(\lambda,\bk), \ubk_{\Gr_\mu})
\]
is a free module generated in even degrees over the ring
\[
\Hom^\bullet(i_*\ubk_{\Gr_\mu}, i_*\ubk_{\Gr_\mu}) \cong H^\bullet(\Gr_\mu).
\]
By Lemma~\ref{lem:even-free}, $i^*j^*\cIstd(\lambda,\bk)$ is even, as desired.
\end{proof}

\begin{thm}[cf.~{\cite[Conjecture~6.3]{mv1}} or~{\cite[Conjecture~13.3]{mv}}]\label{thm:mvconj}
If $p$ is a JMW prime for $\Gv$, then the stalks of $\cIstd(\lambda,\Z)$ have no $p$-torsion.
\end{thm}
\begin{proof}
Let $M$ be a $\Z$-module.  It is a routine exercise to show that if $M$ has $p$-torsion, then $H^i(M \Lotimes \F_p) \ne 0$ for both $i = 0$ and $i = -1$.  Now, let $x \in \Gr$, and consider the stalk $\cIstd(\lambda,\Z)_x$, which is an object in the derived category of finitely generated $\Z$-modules.  Since $\Z$ has global dimension~$1$, $\cIstd(\lambda,\Z)_x$ is isomorphic to the direct sum of its cohomology modules, and if any cohomology module had $p$-torsion, the object $\cIstd(\lambda,\Z)_x \Lotimes \F_p$ would have nonzero cohomology in both even and odd degrees.  But by~\cite[Proposition~8.1(a)]{mv}, we have
\[
\cIstd(\lambda,\Z)_x \Lotimes \F_p \cong \cIstd(\lambda,\F_p)_x,
\]
and Theorem~\ref{thm:jmwconj} tells us that the latter cannot have cohomology in both even and odd degrees.  Thus, $\cIstd(\lambda,\Z)_x$ has no $p$-torsion.
\end{proof}

\appendix
\section{Sheaves on non-locally compact spaces}
\label{app:nonloccpt}

Let $X$ and $Y$ be Hausdorff topological spaces, and let $h: Y \to X$ be a continuous map.  If $X$ is not locally compact, then $h^!$ may not be defined, and $h_!$ may fail to have some familiar properties.  However, if $h$ is an inclusion of a locally closed subset, these difficulties can be largely circumvented.  In this appendix, we briefly explain how to define the functors $h_!$ and $h^!$, and we discuss some of their properties.

Let $\Sh(X)$ denote the category of sheaves of $\bk$-modules on $X$.  The discussion below will make heavy use of the two functors $\Sh(X) \to \Sh(X)$ discussed in~\cite[Proposition~2.3.6 and Definition~2.3.8]{ks}, denoted
\[
\cF \mapsto \cF_U
\qquad\text{and}\qquad
\cF \mapsto \Gamma_U(\cF).
\]
As in the main body of the paper, $h_!$ and $h_*$ will always denote derived functors.  Their non-derived analogues will be denoted by $^\circ h_!$ and ${}^\circ h_*$, respectively.  

As explained in~\cite[Eq.~(2.5.1)]{ks}, the functor $^\circ h_!: \Sh(Y) \to \Sh(X)$ (and hence its derived functor $h_!$) can be defined without any local compactness assumption.  According to~\cite[Proposition~2.5.4]{ks}, when $h$ is a locally closed inclusion, we have a natural isomorphism
\begin{equation}\label{eqn:oh_!-defn}
^\circ h_! h^*\cF \cong \cF_Y.
\end{equation}
On the other hand, inspired by~\cite[Proposition~3.1.12]{ks}, we define a new left-exact functor $^\circ h^!: \Sh(X) \to \Sh(Y)$ by
\begin{equation}\label{eqn:oh^!-defn}
^\circ h^!(\cF) = h^* \Gamma_Y(\cF).
\end{equation}
Let $h^!: \Db(X) \to \Db(Y)$ be its right derived functor.

\begin{lem}\label{lem:h!-basic}
Let $h: Y \to X$ be a locally closed inclusion.  
\begin{enumerate}
\item The functor ${}^\circ h_!$ is exact.\label{it:h!-exact}
\item If $h$ is an open inclusion, then $h^! \cong h^*$.\label{it:h!-open}
\item There are natural isomorphisms $h_! h^*\cF \cong \cF_Y$ and $h^!\cF \cong h^* R\Gamma_Y(\cF)$.\label{it:h!-defn}
\item Let $k: Z \to Y$ be another locally closed inclusion.  There are natural isomorphisms $(h \circ k)_! \cong h_! \circ k_!$ and $(h \circ k)^! \cong k^! \circ h^!$.\label{it:hk!}
\item There are natural isomorphisms $h_!\cF \cong (h_*\cF)_Y$ and $h_*h^!\cF \cong R\Gamma_Y(\cF)$.\label{it:h!-other}
\end{enumerate}
\end{lem}
Comparing part~\eqref{it:h!-defn} with~\cite[Proposition~3.1.12]{ks}, we see that our definition of $h^!$ agrees with the usual one in case $X$ is locally compact.  For part~\eqref{it:hk!}, the usual proofs (see~\cite[Eq.~(2.6.6) and Proposition~3.1.8]{ks}) require the use of $c$-soft sheaves, which may not be available on $X$.
\begin{proof}
\eqref{it:h!-exact}~This is~\cite[Proposition~2.5.4(i)]{ks}.

\eqref{it:h!-open}~We see from~\cite[Proposition~2.3.9(iii)]{ks} that $^\circ h^! \cong h^* \circ {}^\circ h_* \circ h^* \cong h^*$, and so $h^! \cong h^*$ as well.

\eqref{it:h!-defn}~According to~\cite[Propositions~2.3.6]{ks} and part~\eqref{it:h!-exact}, every functor in~\eqref{eqn:oh_!-defn} is exact, so its derived version follows.  In~\eqref{eqn:oh^!-defn}, since $h^*$ is exact, the derived functor of the right-hand side is just $h^*R\Gamma_Y$.

\eqref{it:hk!}~Since $^\circ h_!$ and $^\circ k_!$ are exact, the first assertion follows immediately from its non-derived analogue, found in~\cite[Eq.~(2.5.3)]{ks}.  Next, a routine argument (cf.~\cite[Proposition~2.4.6]{ks}) shows that $^\circ h^!$ takes flabby sheaves to flabby sheaves, and that its derived functor $h^!$ can be computed using flabby resolutions.  Therefore, the desired result follows from its non-derived analogue, which says that
\[
k^* \Gamma_Z h^* \Gamma_Y \cong (h \circ k)^* \Gamma_Z.
\]
(Note that on the left-hand side, $\Gamma_Z$ is a functor on $\Sh(Y)$, while on the right-hand side, it is a functor on $\Sh(X)$.)  This can be checked directly from the definition, in the spirit of~\cite[Proposition~2.3.9]{ks}.

\eqref{it:h!-other}~Since $h^* h_*\cF \cong \cF$, it follows immediately from part~\eqref{it:h!-defn} that $h_!\cF \cong (h_*\cF)_Y$.  For the second assertion, thanks to part~\eqref{it:hk!}, it suffices to consider the cases where $Y$ is open or closed.  If $Y$ is open, the non-derived analogue $^\circ h_* {}^\circ h^! \cong \Gamma_Y$ is~\cite[Proposition~2.3.9(iii)]{ks}, and the derived version follows because $^\circ h^! \cong h^*$ takes flabby sheaves to flabby sheaves.  If $Y$ is closed, one can check from the definitions that ${}^\circ h_* {}^\circ h^! = {}^\circ h_*  h^*  \Gamma_Y \cong \Gamma_Y$.  Since ${}^\circ h_*$ is exact, it follows that $h_* h^! \cong R\Gamma_Y$.
\end{proof}

\begin{lem}
The functor $h^!$ is right adjoint to $h_!$.
\end{lem}
\begin{proof}
It follows from~\cite[Eq.~(2.6.9)]{ks} that $\Hom(\cF_Y, \cG) \cong \Hom(\cF, R\Gamma_Y(\cG))$.  We therefore have the following sequence of natural isomorphisms:
\begin{multline*}
\Hom(h_!(\cF),\cG) \cong \Hom((h_*\cF)_Y,\cG) \cong \Hom(h_*\cF, R\Gamma_Y(\cG)) \cong \\
\Hom(h_*\cF,h_*h^!\cG) \cong \Hom(h^*h_*\cF,h^!\cG) \cong \Hom(\cF,h^!\cG). \qedhere
\end{multline*}
\end{proof}

\begin{lem}\label{lem:openclosed}
Let $i: Z \to X$ be the inclusion of a closed subset, and let $j: U \to X$ be the inclusion of the complementary open subset.  For any $\cF \in \Db(X)$, there are functorial distinguished triangles
\begin{gather*}
i_*i^!\cF \to \cF \to j_*j^*\cF \to, \\
j_!j^*\cF \to \cF \to i_*i^*\cF.
\end{gather*}
\end{lem}
\begin{proof}
According to~\cite[Eq.~(2.6.32)]{ks}, there is a functorial distinguished triangle $R\Gamma_Z(\cF) \to \cF \to R\Gamma_U(\cF) \to$.  Using Lemma~\ref{lem:h!-basic}\eqref{it:h!-other}, we obtain the first distinguished triangle above.  The second follows similarly from~\cite[Eq.~(2.6.33)]{ks}.
\end{proof}

\begin{lem}\label{lem:basechange}
Suppose we have a cartesian square
\[
\begin{tikzpicture}[description/.style={fill=white,inner sep=1.5pt}] 
\matrix (m) [matrix of math nodes, row sep=2.5em,
column sep=2.5em, text height=1ex, text depth=0.25ex, nodes in empty cells]
{ Y' & X' \\
  Y & X \\};
\path[->,>=angle 90, font=\scriptsize]
(m-1-1) edge node[above]{$h'$} (m-1-2)
edge node[left]{$f'$} (m-2-1)
(m-1-2) edge node[right]{$f$} (m-2-2)
(m-2-1) edge node[above]{$h$} (m-2-2);
\end{tikzpicture}
\]
where $h$ and $h'$ are locally closed inclusions.  Then we have $h^!f_* \cong f'_*h'{}^!$.
\end{lem}
\begin{proof}
According to~\cite[Eq.~(2.3.20)]{ks}, we have $\Gamma_Y {}^\circ f_* \cong {}^\circ f_*\Gamma_{Y'}$.  Since all these functors take flabby sheaves to flabby sheaves, we also have the derived version $R\Gamma_Y f_* \cong f_* R\Gamma_{Y'}$.  In other words, $h_*h^!f_* \cong f_* h'_* h'{}^! \cong h_* f'_* h'{}^!$.  Compose with $h^*$ to obtain $h_!f^* \cong f'_*h'{}^!$.
\end{proof}


\end{document}